\documentclass[10pt]{article}
\usepackage{geometry}               % See geometry.pdf to learn the layout options. There are lots.
\usepackage[english]{babel}
\usepackage[latin1]{inputenc}
\usepackage{amssymb,amsmath,amsthm,amsfonts}
\usepackage{graphicx}
\usepackage[usenames,dvipsnames]{xcolor}
\usepackage{cancel}
\usepackage[colorlinks]{hyperref}
\usepackage{filecontents}
\usepackage{dsfont}

\usepackage[sc]{mathpazo}
\newcommand{\N}{\mathbb{N}}

\newcommand{\R}{\mathbb{R}}

\newcommand{\from}{\colon}

\newcommand{\eps}{\varepsilon}

\newcommand{\la}{\lambda}
\newcommand{\loc}{\text{loc}}
\newcommand{\diam}{\text{diam}}
\newcommand{\vfi}{\varphi}

\newcommand{\Id}{\mathrm{Id}}

\newcommand{\Om}{\Omega}

\newcommand{\Hcal}{{\mathcal{H}}}
\newcommand{\Mcal}{{\mathcal{M}}}

\newcommand{\Acal}{{\mathcal{A}}}

\def\XXint#1#2#3{{\setbox0=\hbox{$#1{#2#3}{\int}$ }
\vcenter{\hbox{$#2#3$ }}\kern-.6\wd0}}

\newcommand{\ind}[1]{\mathds{1}_{#1}}

\DeclareMathOperator{\dist}{dist}

\DeclareMathOperator{\supp}{supp}

\newtheorem{proposition}{Proposition}[section]
\newtheorem{theorem}[proposition]{Theorem}
\newtheorem{corollary}[proposition]{Corollary}
\newtheorem{lemma}[proposition]{Lemma}

\theoremstyle{definition}
\newtheorem{definition}[proposition]{Definition}
\newtheorem{remark}[proposition]{Remark}

\newcommand{\beq}{\begin{equation}}
\newcommand{\eeq}{\end{equation}}
\newcommand{\ben}{\begin{enumerate}}
\newcommand{\een}{\end{enumerate}}
\newcommand{\bit}{\begin{itemize}}
\newcommand{\eit}{\end{itemize}}

\newcommand{\NTr}{\ell}

\DeclareMathOperator{\od}{\Lambda}

\title{Singular analysis of the optimizers of the principal eigenvalue in indefinite weighted Neumann problems}

\author{Dario Mazzoleni, Benedetta Pellacci and Gianmaria Verzini}

\begin{document}
\maketitle

\begin{abstract}
We study the minimization of the positive principal eigenvalue associated to a weighted Neumann problem settled in a bounded smooth domain $\Omega\subset \R^{N}$, within a suitable class of sign-changing weights.
Denoting with $u$ the optimal eigenfunction and with 
$D$ its super-level set associated to the optimal weight, 
we perform the analysis of the singular limit of the optimal eigenvalue 
as the measure  of  $D $  tends to zero.
We show that, when the measure of $D$ is sufficiently small, $u $ has a  unique 
local maximum point lying on the boundary of $\Omega$ and $D$ is connected.
Furthermore, the boundary of $D$ intersects the boundary of the box $\Omega$, and more precisely, ${\mathcal H}^{N-1}(\partial D \cap \partial \Omega)\ge C|D|^{(N-1)/N}
$ for some universal constant $C>0$. 
Though widely expected, these properties are still unknown if the measure of $D$ is arbitrary.
\end{abstract}
\noindent
{\footnotesize \textbf{AMS-Subject Classification}}. 
{\footnotesize 49R05, 49Q10, 92D25, 47A75, 35B40.}

\section{Introduction}
A classical model describing the spatial dispersal of a population in an heterogenous 
environment relies on a reaction-diffusion equation of logistic type. In order to represent
the habitat subdivided into patches, one introduces a sign-changing weight $m$ in the equation, 
so that the favorable and hostile zones  respectively correspond to  the positivity and negativity sets of $m$ (see \cite{MR1014659,MR2214420,MR2191264}). The equation is naturally associated  with homogeneous Neumann boundary conditions
when considering this problem in a bounded open set $\Omega\subset \R^{N}$ 
with $\partial \Omega$ acting as a reflecting barrier.

It is well known (see e.g.\ \cite{MR2191264}) that the survival  of the population is 
guaranteed when the principal eigenvalue of the weighted problem
\begin{equation}\label{eq:eigprobOD}
\begin{cases}
 -\Delta u = \lambda m u & \text{in }\Omega\\
 \partial_\nu u = 0 & \text{on }\partial\Omega.
\end{cases}
\end{equation}
is below a certain positive threshold. 
A principal eigenvalue for \eqref{eq:eigprobOD} is a number $\lambda$
having a positive eigenfunction. In case $m^+$ and $m^-$ are
both nontrivial, \eqref{eq:eigprobOD} admits two principal eigenvalues, $0$ and
$\lambda(m)$. Moreover, $\lambda(m)>0$ if and only if $\int_\Omega m <0$, in which case
\begin{equation}\label{eq:def_lambda(m)}
\lambda(m):= \min \left\{\dfrac{\int_\Omega |\nabla u|^2\,dx}{\int_\Omega m u^2\,dx} :  u\in H^1(\Omega),\ \int_\Omega m u^2\,dx>0\right\}.
\end{equation}
Taking into account that the smaller $\lambda(m) $ is, the more chances of  survival 
the population has, a largely studied problem consists in minimizing $\lambda(m)$ 
with respect to the weight or to other relevant parameters of 
the model. The literature is quite extensive, and we refer to the recent papers 
\cite{mazari:hal-02987223,MR3498523,MR3771424,dipierro2021nonlocal} and references therein, for various interesting phenomena ranging from fragmentation to different nonlocal effects.

In this paper, we focus on the minimization of $\lambda(m)$ with respect to the weight. As proved in \cite{ly}, when the mean $\int_\Omega m$ is fixed, as well as lower and upper bounds $-\beta \le m \le 1$, 
the infimum of $\lambda(m)$ is achieved by a
bang-bang (i.e. piecewise constant) optimal weight $m = \ind{D} - \beta\ind{\Omega\setminus D}$,
for some measurable set $D\subset\Omega$. Therefore, one can
equivalently consider the minimization over the class of bang-bang weights $\ind{D} - \beta\ind{\Omega\setminus D}$, under a volume constraint on $D$ in order to fix
the average of $m$. 
\begin{definition}\label{def:od}
Let $\beta>0$. For any $D\subset\Omega$ such that $|D| <\dfrac{\beta|\Omega|}{\beta+1}$
we define the eigenvalue of the corresponding bang-bang weight as
\begin{equation}\label{eq:def_lambda_beta_D}
\lambda(D):=\lambda(\ind{D} - \beta\ind{\Omega\setminus D})= \min \left\{
\dfrac{\int_\Omega |\nabla u|^2\,dx}{\int_D u^2\,dx - \beta \int_{\Omega\setminus D} u^2\,dx} :  
%u\in H^1(\Omega),\ 
\int_D u^2\,dx>\beta \int_{\Omega\setminus D} u^2\,dx\right\}. 
\end{equation}
For $0<\delta<\dfrac{\beta|\Omega|}{\beta+1}$, the optimal design problem for the survival threshold is 
\begin{equation}\label{eq:def_od}
\od(\delta)=\min\Big\{\lambda(D):D\subset \Om,\mbox{ measurable, }|D|=\delta\Big\}.
\end{equation}
\end{definition}

It follows that $\od(\delta)$ is achieved by a set $D_\delta$ (see \cite{MR1014659,ly}), which 
turns out to be the super-level set of a  corresponding principal (positive) eigenfunction
$u_\delta$. In particular, $D_\delta$ contains the maximum points of $u_\delta$. 
As a consequence, both for modeling reasons, and from the mathematical point of view, natural questions arise about the location and the shape of the optimal set $D$. 

This issue is mostly open in its generality, and it is completely understood only in dimension one:
if  $\Omega=(0,1)$, then $u_\delta$ has a unique maximum point, located at the boundary, and $D_\delta$ is either the interval $(0,\delta)$ or $(1-\delta,1)$ (see
 \cite{MR1105497,ly,llnp}, also for the related problem with different boundary conditions).
In particular, in this case the optimal set is connected and its boundary intersects the one of 
$\Omega$. While these properties are expected also in higher dimensions (see 
e.g. \cite[Open Problem 1]{mazari:hal-02987223}), their current understanding, up to our knowledge, is confined to the case in which $\Omega=\Pi_{i=1}^{N}(0,a_{i})$ is an orthotope, a
situation of particular interest because of its natural 
relation with the so-called periodically fragmented environment model (see \cite{MR2214420}).

This case is investigated in \cite{llnp}, where Steiner 
symmetrization arguments are exploited to show that both the principal eigenfunction and the 
optimal weight are non-increasing along each coordinate direction. It follows that the maximum of $u_\delta$ is located at one of the vertices of $\Omega$, 
$D_\delta$ is connected, and $\overline{D}_\delta \cap\partial\Omega$ has positive  $N-1$ Hausdorff measure. 

The main goal of this paper is to study the location and shape of the optimal set for general 
domains,  when the measure $\delta$ is small. Our main result concerning the properties of $D_\delta$ and of the corresponding eigenfunction $u_\delta$ is the following.

\begin{theorem}\label{th:intro_D}
Let $\Omega\subset \R^N$ be an open and bounded set with boundary of class $C^{2,\alpha}$ for some $\alpha\in(0,1)$.
For every $0<\eps<1$ there exists $\delta_{0}>0$ such that, for every $\delta\in (0,\delta_{0})$:
\begin{enumerate}
\item $u_\delta$ has a unique local maximum point $P_\delta\in \partial \Omega$;
\item $D_\delta$ is connected;
\item defining $r_\pm(\delta)$ in such a way that $|B_{r_\pm(\delta)}|=2\delta(1\pm\eps)$, it results
\[
B_{r_-}(P_\delta)\cap\Omega \subset D_\delta \subset B_{r_+}(P_\delta)
\cap\Omega.
\]
In particular, ${\mathcal H}^{N-1}(\partial D_{\delta}\cap \partial \Omega)\ge C\delta^{(N-1)/N}$
for some universal constant $C>0$.
\end{enumerate}
\end{theorem}

Theorem \ref{th:intro_D} provides much information concerning $u_{\delta}$ and 
$D_{\delta}$.
First of all, the uniqueness of the local maximum point $P_{\delta}$, for $\delta$ sufficiently small, immediately implies that $D_{\delta}$ is connected. Secondly, $P_{\delta}\in \partial \Omega$
and $D_{\delta}\subset B_{r_{+}}(P_{\delta})$; 
as a consequence, as $\delta\to0$  both $D_\delta$ and $u_\delta$ concentrate at 
$P_\delta$. In this line, a more precise description of the decay properties of $u_\delta$ is 
obtained in Proposition \ref{prop:2.3} ahead. 
In addition, the third conclusion shows that the shape of $D_{\delta}$, roughly speaking, 
is approximated by the intersection of a ball with $\Omega$. Hence, the $N-1$ 
dimensional Hausdorff measure of the intersection $\partial D_{\delta}\cap \partial 
\Omega$ is positive for $\delta$ sufficiently small (see Remark \ref{rem:Hn-1}). 

It has been discussed for a while in the literature whether $D_\delta$ has a spherical 
shape or not, meaning that $D_\delta \cap \Omega = B \cap \Omega$ for some suitable ball $B$. 
In particular, this was suggested by some numerical 
simulations in \cite{haro}, when $\Omega\subset\R^2$ is a square and $\delta$ is small. 
On the other hand, in \cite{llnp} it is proved that, for general $\Omega$, 
$\partial D_\delta\cap \Omega$ cannot contain portions of spheres. Motivated by this result, 
we devoted some previous papers to analyze the occurrence of spherical 
shapes for $D_\delta$ in some singularly perturbed regime. In particular, in  
\cite{mapeve,mapeve_matrix} we have shown that, for polyhedral 
domains $\Omega$ and $\delta$ fixed, spherical shapes can emerge, centered at the vertex of the smallest solid angle, when the parameter $\beta$ in Definition \ref{def:od} 
diverges to $+\infty$. From this point of view, the last part of Theorem \ref{th:intro_D} 
shows the same phenomenon, for smooth $\Omega$, $\beta$ fixed and $\delta\to0$.

Theorem \ref{th:intro_D} is reminiscent of well-known results in the study of 
singularly perturbed elliptic nonlinear Neumann problems, firstly treated in 
\cite{nitakagi_cpam}. Actually, the general strategy  adopted here is inspired by the one developed by Ni and Takagi. More precisely, we perform a blow-up analysis 
near a maximum point of the eigenfuction $u_\delta$, rescaling the problem in order to 
pass to the limit. On the other hand, our setting is quite different and it requires some new ideas.

Indeed, although our problem is linear, the presence of the sign-changing, piecewise constant weight 
$m_\delta$, depending on the unknown set $D_\delta$, gives rise to several difficulties. First 
of all, in the rescaling procedure we obtain a sequence of weights, which we can only show to 
converge in the weak$^*$-topology of $L^{\infty}_{{\rm loc}}$. Henceforth, 
the optimal regularity for the eigenfunctions $u_\delta$ is  merely $C^{1,\alpha}$, 
obstructing higher order convergence of the blow-up sequences. Furthermore, in order to 
complete our argument, we need to keep track of the behavior of the optimal set during the 
asymptotical analysis; this delicate information is obtained by showing that the boundary of 
$D_\delta$ is trapped in suitable annuli, for $\delta$ small enough (see for more details 
Proposition \ref{prop:nonvanish}). In turn, this property will be fundamental in proving the third 
conclusion of Theorem \ref{th:intro_D}.

As usual, when performing a blow-up analysis it is crucial to detect the natural associated 
limit  problem. In our case, this is provided (up to scaling) by the following one.
\begin{definition}
\[
I_{\mathcal M}:=\inf\left\{\mu(m) : m\in \mathcal M\right\},
\quad \text{where}\quad
\mu(m):=\inf\left\{\frac{\int_{\R^{N}}|\nabla v|^{2}}{\int_{\R^{N}}mv^{2}} : v\in H^{1}(\R^{N}), \int_{\R^{N}}mv^{2}>0\right\}
\]
for $m$ lying in the class
\[
\mathcal M:=\left\{m\in L^{\infty}(\R^{N}): 
-\beta\leq m\leq 1 \text{ a.e.\ in } \R^{N},
\quad\int_{\R^{N}}(m+\beta)\leq 1+\beta 
\right\}. 
\]
\end{definition}
Concerning $I_{\mathcal M}$ we prove the following result.
\begin{theorem}\label{thm:limit_intro}
It results
$$
I_{\mathcal M}=\mu(\ind{B}-\beta \ind{B^{c}})
$$
where $B$ is the ball centered at zero and of measure $1$ and 
$\mu(\ind{B}-\beta \ind{B^{c}})$ 
is achieved by a positive radially 
symmetric eigenfunction 
$w\in C^{1,1}(\R^{N})$ exponentially decaying w.r.t.\ to $r=|x|$.
\end{theorem}
Theorem \ref{thm:limit_intro} allows to conclude the blow-up procedure, in view of the following 
result.
\begin{theorem}\label{thm:bound_intro}
Let $\Omega\subset \R^N$ be an open and bounded set with boundary of class $C^{2,\alpha}$ 
for some $\alpha\in(0,1)$.
As $\delta\to0$
\begin{equation}\label{eq:coeff1}
\od(\delta)=\frac{1}{4^{1/N}}I_{\mathcal M}\, \delta^{-2/N}(1+o(1)).
\end{equation}
Moreover
\begin{equation}\label{eq:upperbound}
\od(\delta)\leq \frac{1}{4^{1/N}} I_{\mathcal M}\, \delta^{-2/N} \left(1-\Gamma \widehat H \delta^{1/N}+o(\delta^{1/N})\right),
\end{equation}
where $\widehat H$ denotes the maximum of the mean curvature of 
$\partial \Omega$, and $\Gamma>0$ is 
a universal constant, independent of $\Omega$ (see equation \eqref{eq:Gamma}). 
\end{theorem}

The proof of Theorem \ref{thm:limit_intro} relies on the following argument: first,
an adaptation of the bathtub principle shows that $I_{\mathcal M}$ is 
achieved by a bang-bang weight $m=\ind{A}-\beta \ind{A^{c}}$; then
symmetrization arguments yield that $A=B$ and finally we prove that $\mu(\ind{B}-\beta \ind{B^{c}})$  is achieved. In demonstrating this last 
step, it is crucial that the optimal weight is positive only in a bounded 
set, to guarantee good compactness properties. 

On the other hand, to prove \eqref{eq:upperbound} in Theorem \ref{thm:bound_intro}, we construct a competitor for $\od(\delta)$: note that 
this amounts to build a suitable $H^{1}(\Omega)$ function and a weight of the form 
$\ind{A}-\beta\ind{A^{c}}$ where $|A|=\delta$. This last requirement makes the argument delicate 
as it forces us to introduce an unknown rescaling factor $r=r(\delta)$. Then, in order to obtain the desired comparison, we need to study the asymptotical  
behavior of $r(\delta)$.
\begin{remark}
We observe that Theorem \ref{th:intro_D} cannot be directly applied to the relevant case of the orthotope, 
which is not $C^{2,\alpha}$. Actually, such smoothness is crucial in the blow-up 
procedure: indeed, as concentration happens near $\partial\Omega$, we need to exploit a 
suitable diffeomorphism to straighten the boundary and extend the solution by reflection.
However, a similar procedure can be used in case $\Omega=(0,1)^{N}$, by more straightforward 
reflection arguments. In particular, we can complement the results in~\cite[Proposition~5]{llnp},
obtaining that
\begin{equation}\label{eq:limite_teo}
\od(\delta)  = \frac14 I_{\mathcal M}\cdot\delta^{-2/N}+o(\delta^{-2/N})
\qquad\text{as }\delta\to0
\end{equation}
and that, defining $r_\pm(\delta)$ in such a way that $|B_{r_\pm(\delta)}|=2^N\delta(1\pm\eps)$, it results
\[
B_{r_-}(0)\cap\Omega \subset D_\delta \subset B_{r_+}(0)
\cap\Omega.
\]
\end{remark}

In the light of  the above considerations, one can wonder, when $\Omega$ is smooth, where the concentration points $P_{\delta}$ accumulate, as $\delta\to0$. A natural 
conjecture is that this should happen at points of maximal mean curvature, 
and there is a number of clues in this direction. First, as we mentioned, 
this is the case when $\Omega$ is an orthotope; secondly, the maximal mean curvature appears in the upper bound 
\eqref{eq:upperbound}; moreover, we obtained 
strong indications of such behavior in \cite{mapeve}, dealing 
with the double asymptotic $\beta\to+\infty $ and $\delta\to0$.
Actually, in the semilinear case, Ni and Takagi proved this property in \cite{nitakagi_duke},
by showing that the upper bound analogous to \eqref{eq:upperbound} is actually 
an exact expansion of the critical level. To this aim, the crucial step was a 
sharp and deep study of a linearized equation associated to the problem they 
study.
From this point of view, it is not clear how to extend such ideas in our context, 
as this should involve a ``linearization'' of the optimal set $D_\delta$.

The paper is organized as follows. In Section~\ref{sec:casoRN} we study the limit problem,
proving all the results involving $I_{\mathcal M}$ and in particular Theorem 
\ref{thm:limit_intro}. Section~\ref{sec:bfa} is devoted to the proof of the bound from 
above in Theorem \ref{thm:bound_intro} (equation \eqref{eq:upperbound}) and in 
Section~\ref{sec:blowup} we complete the proof of Theorems   
\ref{th:intro_D} and \ref{thm:bound_intro}.

\bigskip
\textbf{Notation.}
\begin{itemize}
\item $|\cdot|$ denotes the Lebesgue $N$ dimensional measure and $\Hcal^{N-1}(\cdot)$ the Hausdorff $N-1$ dimensional measure.
\item For a function $f$, its positive/negative parts are denoted as 
$f^\pm(x)=\max\{\pm f(x) , 0\}$.
\item The characteristic function of a set $E$ is denoted by $\ind{E}$. 
\item $B_r(x)$ denotes the ball of radius $r>0$ centered at $x\in \R^N$. If $x=0$, we often write $B_r=B_r(0)$. On the other hand, $B$ is the ball centered at the origin and with $|B|=1$.
\item We call $\omega_N = |B_1|$ the measure of a ball of radius $1$.
\item $B^{+}_{r}=B_{r}\cap \R^{N}_{+}$.
\item $H_P$ denotes the mean curvature of $\partial \Omega$ at $P\in\partial\Omega$, and $\widehat H = \max_{P\in\partial\Omega} H_P$.
\item For a sequence $(\delta_k)_k$, we write $P_k=P_{\delta_k}$, $w_k=w_{\delta_k}$, and so on.
\item $C,C_1,C',\dots$ denote any (non-negative) universal constant, which may also change 
from line to line.
\end{itemize}

\section{A spectral optimal design problem in 
\texorpdfstring{$\R^N$}{R N}}\label{sec:casoRN}

In this section we consider the minimization problem
\begin{equation}\label{eq:defIM}
I_{\mathcal M_k}:=\inf{\Big\{\mu(m) : m\in \mathcal M_k\Big\}},
\end{equation}
for the weighted eigenvalue
\begin{equation}\label{eq:defmu}
\mu(m):=\inf\left\{\frac{\int_{\R^N}|\nabla v|^2}{\int_{\R^N}mv^2} : v\in H^1(\R^N),\;\int_{\R^N}mv^2>0\right\},
\end{equation}
with $m$ a sign-changing weight belonging to the class
\begin{equation}\label{eq:defMk}
\mathcal M_{k}%=\mathcal M(\beta)
:=\left\{m\in L^\infty(\R^N) : 
\begin{aligned}
&-\beta\leq m\leq 1 \text{ a.e. in $\R^N$, }\\
&\int_{\R^N}(m+\beta)\leq k(1+\beta)
\end{aligned}
\right\},
\end{equation}
where $k>0$, $\beta>0$ are fixed constants. Notice that $m\le 0$ a.e.\ implies $\mu(m)=+\infty$, thus the minimization can be restricted to the weights satisfying 
$|\{x\in \R^{N} : m>0\}|>0$.  
Morever, in general, if $m\in \mathcal M$, then
$m\not\in L^1(\R^N)$. Nonetheless, the auxiliary nonnegative weight $\widetilde m=m+\beta$ belongs to $L^{1}(\R^{N})$.
If $m=\ind{E}-\beta\ind{E^c}$ is bang-bang, then with a slight abuse of notation we write $\mu(m)=\mu(E)$.

Actually, the parameter $k$ can be easily scaled out, 
as the following remark shows.
\begin{remark}\label{rmk:scaling}
We notice that
\[
m\in \Mcal_{k} 
\qquad\iff\qquad
m_t(x):=m(t^{-\frac1N}x) \in \Mcal_{tk}.
\]
Furthermore, for $v\in H^1(\R^N)$ and $v_t(x):= v(t^{-1/N}x)$,
\[
\frac{\int_{\R^N}|\nabla v_t|^2}{\int_{\R^N}m_tv_t^2} = t^{-\frac2N} 
\frac{\int_{\R^N}|\nabla v|^2}{\int_{\R^N}mv^2} . 
\]
We deduce that
\[
I_{\mathcal M_{k_2}}=\left(\frac{k_2}{k_1}\right)^{-\frac2N}I_{\mathcal M_{k_1}}.
\]
Moreover, in case such values are achieved, the optimal weights and eigenfunctions 
scale as well. More precisely, 
\[
w_{[1]}\text{ achieves }I_{\Mcal_1}  
\qquad\iff\qquad
w_{[k]}(x):=w_{[1]}(k^{-\frac1N}x) \text{ achieves } I_{\Mcal_{k}}.
\]
We notice that, with this notation, equations \eqref{eq:coeff1} and \eqref{eq:limite_teo} entail respectively
\[
\od(\delta) \sim I_{{\Mcal}_2}\delta^{-2/N}\qquad\text{ and }
\od(\delta) \sim I_{{\Mcal}_{2^N}}\delta^{-2/N}.
\]
\end{remark}
In view of the previous remark, in the whole paper, when $k=1$ we drop the dependence on $k$ in the notation, as we mostly work with $I_{\mathcal M_{1}}=I_{\mathcal M}$. Our goal is to prove the following result. Analogously, we write $w=w_{[1]}$.
\begin{theorem}\label{thm:limit_pb}
The value $I_{\Mcal}$ is achieved, uniquely up to translations, by the weight
\[
m(x) = \ind{B} - \beta\ind{B^c},
\]
where $B$ denotes the ball of unit measure, with an associated positive eigenfunction $w\in C^{1,1}(\R^N)$ solving $-\Delta w = I_\Mcal m w$. Namely, $I_{\mathcal M}=\mu(B)$. Moreover, 
$w$ is radially symmetric, decreasing in $r=|x|$, and such that
\begin{equation}\label{eq:exp_decay}
w(r) = C_1 r^{-\frac{N-1}2}e^{-\sqrt{\mu\beta}r}\left(1+O\left(r^{-1}\right)\right),
\qquad w'(r) = C_2 r^{-\frac{N-1}2}e^{-\sqrt{\mu\beta}r}\left(1+O\left(r^{-1}\right)\right)
\end{equation}
as $r\to+\infty$, for suitable constants $C_1,C_2$.
\end{theorem}

The remaining part of this section is devoted to the proof of Theorem  \ref{thm:limit_pb}.

The first step, quite standard in this type of problems, is to reduce to bang-bang weights. 
To this aim, we use the so called \emph{bathtub principle}, see e.g.~\cite[Theorem~1.14]{lilo}. Since here we need a formulation which is slightly different from the usual one, we provide a proof.
\begin{proposition}[bathtub principle]\label{prop:bathtub}
Let $f\in L^1(\R^N)$ be a nonnegative function. 
Then, the problem\[
\sup_{m\in \mathcal M} \int_{\R^N}mf,
\] 
is solved by a weight $m_{o}(x)=\ind{D}(x)-\beta\ind{D^c}(x)$, for a measurable set $\{f> t\}\subset D\subset \{f\geq t\}$,
with \[
t:=\inf\Big\{s\in \R : |\{f> s\}|\leq 1\Big\}
\qquad
\text{and}
\qquad
|D|=1.
\]
\end{proposition}
\begin{proof}
The intuitive idea of the bathtub principle is to consider a weight of the form $m(x)=\ind{\{f>t\}}(x)-\beta\ind{\{f\leq t\}}(x)$, with \[
t:=\inf\Big\{s\in \R : |\{f> s\}|\leq 1\Big\}.
\]
If $|\{f>t\}|<1$, we need to take a set $A\subset \{f= t\}$ such that 
\[
|\{f>t\}\cup A|=1.
\]
To check that the choice of such a set $A$ is possible, it is enough to note that, by the definition of $t$ as an infimum, for all $\vartheta>0$\[
|\{f>t-\vartheta\}|>1,\qquad \text{hence}\qquad |\{f>t-\vartheta\}|-|\{f>t\}|\geq 1-|\{f>t\}|,
\]
and passing to the limit as $\vartheta\to 0$, we infer that \[
|\{f=t\}|\geq 1-|\{f>t\}|,
\]
so that an appropriate set $A$ exists. On the other hand, if $|\{f>t\}|=1$, then $\{f>t\}$ is 
already a good candidate and we choose $A=\emptyset$. In both cases, we define 
\[
D:=\{f>t\}\cup A,\qquad m_{o}(x)=\ind{D}(x)-\beta\ind{D^c}(x),\quad x\in \R^N.
\]

Recalling \eqref{eq:defMk} it is easy to check that $m_{o}\in \mathcal M$, as the measure constraint on $\{m_{o}>0\}$ follows from the definition (this also implies that the weight is sign-changing), as well as the bounds from above and from below. Moreover, the integral constraint 
\[
\int_{\R^N}(m_{o}+\beta)=1+\beta,
\]
is satisfied as well.

Finally, we check that $m_{o}$ is actually an optimal weight. To do this, we use the layer-cake representation (Talenti formula) and Fubini theorem, to write \[
\int_{\R^N}m(x)f(x)\,dx=\int_{0}^{+\infty}\left(\int_{\R^N}\ind{\{f>s\}}(x)m(x)\,dx\right)\,ds.
\]
Then the claim follows if we prove that for almost all $s>0$,
\begin{equation}\label{eq:layercakestima}
\int_{\R^N}\ind{\{f>s\}}(x)m(x)\,dx\leq \int_{\R^N}\ind{\{f>s\}}(x)m_{o}(x)\,dx,\qquad \text{for all }m\in \mathcal M.
\end{equation}
We note that, if $s>t$, 
it results $\{f>s\}\subset \{f>t\}\subset  D$ and,  as $m_{o}=1$ on $D$,
we get
\[
\int_{\R^N}\ind{\{f>s\}}(x)m(x)\,dx\leq \int_{\R^N}\ind{\{f>s\}}\,dx=\int_{\R^N}\ind{\{f>s\}}(x)m_{o}(x)\,dx,\qquad \text{for all }m\in \mathcal M.
\]
On the other hand, if $s<t$, one needs to be more careful. First of all, for every 
$m\in \mathcal M$ we write $\widetilde m=m+\beta$, so that $\widetilde m$ is a nonnegative 
function belonging to $L^{1}(\R^{N})$. We observe that, as $f$ is $L^1$, $|\{f>s\}|<+\infty$ for all $s>0$. Then proving 
\[
\int_{\R^N}\ind{\{f> s\}}(x)\widetilde m_{o}(x)\,dx\geq \int_{\R^N}\ind{\{f> s\}}(x)\widetilde m(x)\,dx,
\]
is equivalent to prove 
\[
\int_{\R^N}\ind{\{f> s\}}(x) m_{o}(x)\,dx\geq \int_{\R^N}\ind{\{f> s\}}(x) m(x)\,dx,
\]
for all $m\in \mathcal M$.
We first observe that \[
\begin{split}
\int_{\R^N}\widetilde m_{o}=(1+\beta)|D|= 1+\beta\geq \int_{\R^N}\widetilde m,\qquad \text{for all }\widetilde m\in \mathcal M+\beta,\\
\int_{\R^N}\ind{\{f\leq s\}}\widetilde m_{o}=0\leq \int_{\R^N}\ind{\{f\leq s\}}\widetilde m,\qquad \text{for all }\widetilde m\in \mathcal M+\beta.
\end{split}
\]
All in all, we have \[
\begin{split}
&\int_{\R^N}\ind{\{f> s\}}(x)\widetilde m_{o}(x)\,dx=\int_{\R^N}\widetilde m_{o}-\int_{\R^N}\ind{\{f\leq s\}}\widetilde m_{o}\\
&\geq \int_{\R^N}\widetilde m-\int_{\R^N}\ind{\{f\leq s\}}\widetilde m=\int_{\R^N}\ind{\{f> s\}}(x)\widetilde m(x)\,dx,\qquad \text{for all }\widetilde m\in \mathcal M+\beta.
\end{split}
\]
Putting all the information above together, we conclude that \[
\begin{split}
&\int_{\R^N}m(x)f(x)\,dx=\int_{0}^{+\infty}\left(\int_{\R^N}\ind{\{f>s\}}(x)m(x)\,dx\right)\,ds\\
&\leq \int_{0}^{+\infty}\left(\int_{\R^N}\ind{\{f>s\}}(x)m_{o}(x)\,dx\right)\,ds=\int_{\R^N}m_{o}(x)f(x)\,dx,
\end{split}
\]
for all $m\in \mathcal M$, and the proof is finished.
\end{proof}
We can now show that the minimization in \eqref{eq:defIM} is equivalent to the minimization 
among bang-bang weights. Introducing the class of admissible favorable sets 
\[
\mathcal A:=\Big\{A\subset \R^N : \text{$A$ is measurable and }0<|A|\leq 1\Big\},
\]
we note that the optimal set $D$ provided by the bathtub principle is contained in the class $\mathcal A$; on the other hand, the weight $\ind{A}-\beta\ind{A^c} \in \mathcal M$ for every $A\in\Acal$. With a slight abuse of notation, we write 
\[
\mu(A)=\mu(\ind{A}-\beta\ind{A^c}),\qquad \text{for all }A\in \mathcal A.
\]
\begin{lemma}\label{le:infbangbang}
We have  
\[
I_{\mathcal M}=I_{\mathcal A}:=\inf{\Big\{\mu(A) : A\in \mathcal A\Big\}}.
\]
\end{lemma}
\begin{proof}
First of all, we notice that the claim can be rewritten as an equality of two \emph{inf-inf}:\[
\begin{split}
&\inf{\left\{\inf\left\{\frac{\int_{\R^N}|\nabla v|^2}{\int_{\R^N}mv^2} : v\in H^1(\R^N),\;\int_{\R^N}mv^2>0\right\} : m\in \mathcal M\right\}}\\
&=\inf{\left\{ \inf\left\{\frac{\int_{\R^N}|\nabla v|^2}{\int_{\R^N}(\ind{A}-\beta\ind{A^c})v^2} : v\in H^1(\R^N),\;\int_{\R^N}(\ind{A}-\beta\ind{A^c})v^2>0\right\}: A\in \mathcal A\right\}}.
\end{split}
\]
Since $\ind{A}-\beta\ind{A^c} \in \mathcal M$ for all $A\in {\mathcal A}
$, it is clear that $I_{\mathcal M}\leq I_{\mathcal A}$, hence we can focus on the opposite inequality.
For any $\eps>0$ we can find $m_\eps\in \mathcal M$ and $\psi_\eps\in H^1(\R^N)$ with $\int_{\R^N}m_\eps\psi_\eps^2>0$, such that 
\[
I_{\mathcal M}\geq \frac{\int_{\R^N}|\nabla \psi_\eps|^2}{\int_{\R^N}m_\eps\psi_\eps^2}-\eps.
\]
Then, thanks to the bathtub principle, we have \[
\frac{\int_{\R^N}|\nabla \psi_\eps|^2}{\int_{\R^N}m_\eps\psi_\eps^2}\geq \frac{\int_{\R^N}|\nabla \psi_\eps|^2}{\displaystyle{\sup_{m\in \mathcal M}}\textstyle\int_{\R^N}m\psi_\eps^2}=\frac{\int_{\R^N}|\nabla \psi_\eps|^2}{\int_{\R^N}(\ind{D}-\ind{D^c})\psi_\eps^2},
\]
for some $D\in \mathcal A$.
Noting that  
\[
\int_{\R^N}(\ind{D}-\ind{D^c})\psi_\eps^2\geq \int_{\R^N}m_\eps\psi_\eps^2>0,
\]
we can infer 
\[
I_{\mathcal M}\geq \frac{\int_{\R^N}|\nabla \psi_\eps|^2}{\int_{\R^N}(\ind{D}-\ind{D^c})\psi_\eps^2}-\eps\geq \mu(D)-\eps\geq I_{\mathcal A}-\eps,
\]
and since $\eps$ is arbitrary we conclude the proof.
\end{proof}

In order to solve the minimization problem, it is then enough to work on the case of bang-bang weights, where the Schwarz symmetrization comes to our rescue.
\begin{lemma}\label{lem:pallaottimapblimite}
We have 
\[
I_{\mathcal A}=\mu(B).
\]
\end{lemma}
\begin{proof}
The proof of this fact is based on the Schwarz symmetrization. 
Let $D\in \mathcal A$ and $m_{D}:=\ind{D}-\beta\ind{D^c}$. For any $\eps>0$ 
we can find $\psi_\eps\in H^1(\R^N)$ with 
\[
\int_{\R^N}m_{D}\psi_\eps^2>0
\qquad
\text{and}\qquad
\mu(D)\geq \frac{\int_{\R^N}|\nabla \psi_\eps|^2}{\int_{\R^N}m_{D}\psi_\eps^2}-\eps.
\]
We denote by $(D^*,\psi_{\eps}^*)$ the Schwarz rearrangement of $(D,\psi_\eps)$. 
Since $m_D$ is piecewise constant, its Schwarz 
rearrangement may be defined as
\begin{equation}\label{eq:defmstar}
m_{D}^*:=\ind{D^*}-\beta\ind{(D^*)^c}=m_{D^{*}}= (m_{D}+\beta)^{*}-\beta\in\Mcal.
\end{equation}
By the P\'olya-Szeg\"o inequality, we have \[
\int_{\R^N}|\nabla \psi_\eps|^2\geq \int_{\R^N}|\nabla \psi_{\eps}^*|^2.
\]
On the other hand, the denominator in the Rayleigh quotient is a little more complicated to treat.
First of all, we apply the Riesz rearrangement inequality~\cite[Theorem~3.4]{lilo} to $m_{D}+\beta$ and $\psi_\eps^2$, which are admissible as they are nonnegative and their positive superlevels have finite measure.
This and \eqref{eq:defmstar} entail
\[
\int_{\R^N}(m_{D}+\beta)\psi_\eps^2\leq \int_{\R^N}(m_{D}+\beta)^*\left(\psi_\eps^2\right)^*=\int_{\R^N}(m_{D}^*+\beta)\left(\psi_{\eps}^*\right)^2,
\]
where we used the properties of the Schwarz rearrangement.
Since $\|\psi_\eps\|_{L^2}=\|\psi_\eps^*\|_{L^2}$, \eqref{eq:defmstar} implies 
\[
\int_{\R^N} m_{D}\psi_{\eps}^2\leq \int_{\R^N} m_{D}^*\left(\psi_{\eps}^*\right)^2=\int_{\R^N} m_{D^{*}}\left(\psi_{\eps}^*\right)^2,
\]
yielding
\[
\mu(D)\geq\frac{\int_{\R^N}|\nabla \psi_\eps|^2}{\int_{\R^N} m_{D}\psi_\eps^2}
-\eps\geq \frac{\int_{\R^N}|\nabla \psi_{\eps}^*|^2}{\int_{\R^N} m_{D^*}\!\left(\psi_{\eps}^*\right)^2}-\eps\geq \mu(D^*)-\eps\geq \mu(B)-\eps,
\]
and the conclusion follows since $\eps>0$ and $D\in\Acal$ are arbitrary.
\end{proof}
Next we show that $\mu(B)$ is achieved. Actually, this can be done for any  bounded open set  $E\subset \R^N$.
\begin{lemma}\label{le:existeigenfunction}
Let $E\subset \R^N$ be an open and bounded set, $E\in \mathcal A$, and $m_E:=\ind{E}-\beta\ind{E^c}\in \mathcal M$. There exists an eigenfunction $w\in H^1(\R^N)$ corresponding to the principal eigenvalue $\mu(E)=\mu(m_E)$, that is,
\begin{equation}\label{eq:infmu}
\mu(E)=\inf\left\{\frac{\int_{\R^N}|\nabla v|^2}{\int_{\R^N}m_Ev^2} : v\in H^1(\R^N),\;\int_{\R^N}m_Ev^2>0\right\}=\frac{\int_{\R^N}|\nabla w|^2}{\int_{\R^N}m_E w^2}.
\end{equation}
\end{lemma}
\begin{proof}
Taking any $v\in H^1_0(E)$, it is immediate to see that $\mu(E)<+\infty$. Let $w_n$ be a minimizing sequence for~\eqref{eq:infmu}. Without loss of
generality we can suppose that
\[
0<\int_{\R^N}m_E w_n^2\le \int_{E}w_n^2=1,\qquad\text{for all }n\in \N.
\]
Then it is easy to check that
\[
\int_{E^c}w_n^2\leq \frac{1}{\beta},\qquad \int_{\R^N}|\nabla w_n|^2 \le (\mu(E)+1)\int_{\R^N}m_E w_n^2\le \mu(E)+1,
\qquad \text{for $n$ large}.
\]
Hence\[
1\leq \|w_n\|^2_{L^2(\R^N)}\leq 1+\frac{1}{\beta}\qquad \text{and}\qquad \|w_n\|^2_{H^1(\R^N)}\leq C,\qquad \text{for all }n\in\N,
\]
for some constant $C>0$ independent of $n$.
Therefore, passing to a (nonrelabeled) subsequence, we have \[
w_n\rightharpoonup w,\qquad\text{weakly in }H^1(\R^N),\text{ and strongly in }L^2(E).
\]
Hence, $w\not\equiv 0$, as \[
1=\lim_{n\to +\infty}\int_{E}w_n^2=\int_Ew^2.
\]
On the other hand, by lower semicontinuity of the norm with respect to the weak convergence, \[
\int_{E^c}w^2\leq \liminf_{n\to+\infty}\int_{E^c}w_n^2.
\]
All in all, \[
\int_{\R^N}m_E w^2\geq \lim_{n\to+\infty}\int_E w_n^2-\beta\liminf_{n\to+\infty}\int_{E^c}w_n^2\geq \limsup_{n\to+\infty}\int_{\R^N} m_E w_n^2.
\]
The weak lower semicontinuity of the $L^{2}$ norm of the gradient allows us to conclude
\[
\frac{\int_{\R^N}|\nabla w|^2}{\int_{\R^N}m_E w^2}\leq\liminf_{n\to+\infty}\frac{\int_{\R^N}|\nabla w_n|^2}{\int_{\R^N}m_E w_n^2}=\mu(E),
\]
so the claim is proved.
\end{proof}
\begin{remark}\label{rem:palle}
Notice that, by the equation, the critical set of any optimal eigenfunction has zero measure. 
Taking into account the characterization of the equality cases in the P\'olya-Szeg\"o
inequality \cite{MR929981}, one deduces that the ball is the unique minimizer for 
$I_{\mathcal M}$, and its principal eigenfunction is radial.
\end{remark}
To conclude the proof of Theorem  \ref{thm:limit_pb}, we study the decay of the optimal eigenfunction at infinity.
\begin{lemma}\label{le:decaypsialto}
Let $w=w(r)$ be the principal eigenfunction associated to $\mu(B)$. Then there exist $C$ 
such that \eqref{eq:exp_decay} is satisfied.
\end{lemma}
\begin{proof}
Since $m_{B}$ is piecewise constant and $w=w(r)$ is a radial $H^1(\R^N)$-function, we have that $w$ solves
\[
\begin{cases}
r^2 w_{rr}+(N-1)r w_{r}-\mu \beta r^2 w = 0 & \text{for }r> r_0,\\
w(+\infty)=0,
\end{cases}
\]
where $r_0$ is the radius of $B$. Writing 
\[
w(r) = r^{-\frac{N}{2}+1} \widetilde w(r\sqrt{\mu\beta}),
\]
we have that $\widetilde w$ solves
\[
\begin{cases}
r^2 \widetilde w_{rr}+r\widetilde w_{r}-\left(\left(\frac{N}{2}-1\right)^2+r^2\right)\widetilde w = 0 & \text{for }r> r_0,\\
\widetilde w(+\infty)=0.
\end{cases}
\]
We deduce that 
\[
\widetilde w(r) = C K_{\frac{N}2-1}(r),
\]
where $K_{\nu}$ is the modified Bessel function of the second kind, with parameter $\nu$. The
lemma follows by well known decay properties of $K_\nu$, see e.g. \cite[p. 5,9,23--24]{MR0058756vol2}.
\end{proof}
\section{Bound from above}\label{sec:bfa}

We aim to prove the following result.
\begin{theorem}\label{th:boundabove}
For any $P\in \partial \Omega$, we have that 
\begin{equation}\label{eq:boundabove}
\od(\delta)\leq 2^{-2/N}I_{\Mcal}\,\delta^{-2/N}\left(1-2^{1/N}\frac{2\alpha\gamma}{\int_{\R^N_+}|\nabla w|^2}\delta^{1/N}+o(\delta^{1/N})\right)
\end{equation}
where $w$ achieves $I_{\Mcal}$ (see Theorem \ref{thm:limit_pb} and Remark \ref{rmk:scaling}) 
and
\begin{equation}\label{eq:alfagamma}
\alpha=(N-1)H_P,\qquad\gamma:=\frac1{N+1}\int_{\R^{N}_{+}}|\nabla w|^2 z_{N}\, dz
\end{equation}
and $H_P$ denotes the mean curvature of $\partial \Omega$ at the point $P$. 
\end{theorem}
\begin{remark}
Since $\widehat H = \max_{P\in\partial\Omega} H_P$, the bound from above \eqref{eq:upperbound} in Theorem \ref{thm:bound_intro} follows at once, with
\begin{equation}\label{eq:Gamma}
\Gamma = \frac{2^{1+1/N}(N-1)}{N+1}\,\frac{\int_{\R^N_+} w'(|z|)^2 z_N\,dz}{\int_{\R^N_+} w'(|z|)^2\,dz}
\end{equation}
(recall that $w$ is radial).
\end{remark}
\begin{remark}\label{rem:M1M2}
Recalling Remark \ref{rmk:scaling}, Theorem \ref{th:boundabove} yields
\[
\limsup_{\delta \to 0} \od(\delta)\cdot \delta^{-2/N} \leq 2^{-2/N}I_{\Mcal}=I_{\Mcal_2}.
\]
\end{remark}

To prove Theorem \ref{th:boundabove}, in the spirit of \cite[Sect. 3]{nitakagi_cpam} we use a diffeomorphism to flatten the boundary of $\Omega$ near a suitable point. To this aim, 
we introduce some notation which will be used in the following.

Let $\Omega\subset \R^N$ be a $C^{2,\alpha}$ domain. Up to an affine change of variables, 
we can assume that $P=0\in \partial \Omega$ and that the outer unit normal to the boundary 
of $\Omega$ is $-e_N$.
Then, using the notation\[
x'=(x_1,\dots,x_{N-1}),
\]
there is $\delta_0>0$, a $C^{2,\alpha}$ function \[
\psi\colon \Big\{x'\in \R^{N-1}: |x'|<\delta_0\Big\}\to \R,
\]  
and a neighborhood of the origin $\mathcal N$ such that 
\begin{enumerate}
\item[i)]$\psi(0)=0$, $\nabla \psi(0)=0$, $\Delta \psi(0)=(N-1)H_0 = \alpha$,
\item[ii)] $\displaystyle
\partial \Omega\cap \mathcal N=\Big\{(x',x_N) : x_N=\psi(x')\Big\},\qquad \Omega\cap \mathcal N=\Big\{(x',x_N) : x_N>\psi(x')\Big\}.
$
\end{enumerate}
For a certain $\delta_1>0$, we define a diffeomorphism \[
\Phi\colon \Big\{y\in \R^N : |y|\leq \delta_1\Big\}\to \R^N,\qquad x=\Phi(y)=(\Phi_1(y),\dots, \Phi_N(y)),
\]
as 
\[
\Phi_j(y)=
\begin{cases}
y_j-y_N\frac{\partial \psi}{\partial x_j}(y'),&\qquad \text{for }j=1,\dots, N-1,
\smallskip\\
y_N+\psi(y'),&\qquad \text{for }j=N.
\end{cases}
\]
We note that $D\Phi(0)=\Id$, due to the properties of $\psi$, and therefore $\Phi$ is locally invertible in, say, 
$B_{3\NTr}$ for some $\NTr>0$.  We define, for $j=1,2,3$,
\begin{equation}\label{eq:defPsi}
D_j=\Phi(B_{j\NTr}^+)\subset\Omega,\quad\text{and}\quad \Psi\from D_3\to B_{3\NTr}^+,\ \Psi(x):=\Phi^{-1}(x).
\end{equation}
The map $\Psi$ can be seen as a local diffeomorphism straightening the boundary around $0\in \partial \Omega$. For future reference, we remark that
\begin{equation}\label{eq:lemmaA1}
\begin{split}
\det D\Phi (y)  &= 1 -\alpha y_N + O(|y|^2),\smallskip\\
\left| \frac{y}{|y|}D\Psi (\Phi(y))\right|^2  &= 1 +2 y_N \sum_{i,j=1}^{N-1}  
\psi_{ij}(0)\frac{y_iy_j}{|y|^2} + O(|y|^2),
\end{split}
\qquad\text{as  }y\to0,
\end{equation}
see \cite[Lemma A.1]{nitakagi_cpam}.

To prove Theorem \ref{th:boundabove}, we build a competitor for $\od(\delta)$ by composing the diffeomorphism $\Psi$ 
with a suitable dilation of the  weight 
\[
m(x)=\ind{B}-\beta\ind{\R^N\setminus B},
\]
and of the corresponding eigenfunction $w$ obtained in Theorem \ref{thm:limit_pb}. A main difference with 
respect to \cite{nitakagi_cpam} is that we have to keep track of the measure of the positivity set of the 
weight. Let us define
\[
m_\delta(x)=
\begin{cases}
m(\Psi(x)/r(\delta)),\qquad &\text{if }x\in D_2,\\
-\beta,\qquad &\text{if }x\in \Omega\setminus D_2,
\end{cases}
\]
and $r(\delta)$ in such a way that the weight $m_\delta$ is admissible, that is, \[
|\{x\in \Omega : m_\delta(x)=1\}|=\delta.
\]
For $\delta$ small, the asymptotic relation between $\delta$ and $r(\delta)$ is explicit.
\begin{lemma}\label{le:rdelta}
It holds $r(\delta)\to0$ and
\begin{equation}\label{eq:rdelta}
\delta=r^N(\delta)\left(\frac12-\frac{1}{N+1}\omega_{N-1}\omega_N^{-\frac{N+1}{N}}\alpha\,r(\delta)+O(r^2(\delta))\right),
\qquad\text{as $\delta\to 0$.}
\end{equation}
\end{lemma}
\begin{proof}
We write $r=r(\delta)$. We have
\[
\{x\in \Omega : m_\delta(x)=1\} = \left\{x\in D_2 : \dfrac{\Psi(x)}{r}\in B\right\}
=\left\{\Phi(y) : y\in B_{2\NTr}^+\cap r B\right\}.
\]
In particular, since this set has measure $\delta$, we have that $B_{2\NTr}^+\cap r B = r B^+$ 
for $\delta$ sufficiently small. As consequence, $r\to0$ as $\delta\to0$. Using also~\eqref{eq:lemmaA1}, we compute\[
\begin{split}
\delta&=|\{x\in \Omega : m_\delta(x)=1\}|=\int_{rB^+}\det D\Phi(y)\,dy=\int_{rB^+} 
\Big(1-\alpha y_N+O(|y|^2)\Big)\,dy\\
&=r^N\int_{B^+}\Big(1-\alpha rz_N+O(r^2|z|^2)\Big)\,dz=r^N\Big(\frac12-\alpha \frac{\omega_{N-1}}{N+1}\omega_N^{-\frac{N+1}{N}}r+O(r^2)\Big),
%=r(\delta)^N\Big(1-H_P \frac{N-1}{N+1}\omega_{N-1}\overline R^{N+1}r(\delta)+O(r(\delta)^2)\Big),
\end{split}
\]
where $\alpha=\Delta \psi(0)$ and we have also used the fact that 
\begin{equation*}%\label{eq:intzN}
\int_{B^+}z_N\,dz=\frac{\omega_{N-1}}{N+1}\omega_N^{-\frac{N+1}{N}}. \qedhere
\end{equation*}
\end{proof}
Turning to the eigenfunction $w$, to build a competitor $\varphi_\delta$ after rescaling we also
need to cut-off. We define
\[
\varphi_\delta(x)=\zeta_{\NTr}(|\Psi(x)|)w\left(\frac{\Psi(x)}{r(\delta)}\right),
\qquad\text{where}\quad
\zeta_\rho(t)=\begin{cases}
1 & 0\leq t\leq \rho
\\
2-\rho^{-1}t & \rho< t\leq 2\rho
\\
0 &   t> 2\rho.
\end{cases}
\]
For easier notation, it is convenient to introduce also the function 
$w_{*}(z):=\zeta_{\NTr/r(\delta)}(|z|)w(z)$, in such a way that
\[
\varphi_{\delta}(x) = w_*\left(\frac{\Psi(x)}{r(\delta)}\right).
\]
Notice that, in principle, both $\varphi_\delta$ and $w_*$ are not defined in the whole $\R^N$; 
nonetheless, by trivial extension, we can assume that they are Lipschitz and compactly supported on 
$\R^N$. 
\begin{proposition}\label{prop:boundabove}
It holds, as $\delta \to 0$,
\[
\begin{split}
\int_\Om|\nabla\varphi_\delta|^2&= r^{N-2}(\delta)\left\{\frac12\int_{\R^N}|\nabla w|^2-(N-1)\alpha\gamma r(\delta)+O(r^2(\delta))\right\},\\
\int_{\Omega}m_\delta \varphi_\delta^2&=r^N(\delta)\left\{\frac12\int_{\R^N}mw^2-\alpha \gamma_1 r(\delta)+O(r^2(\delta))\right\},
\end{split}
\]
where $\alpha$ and $\gamma$ are defined in \eqref{eq:alfagamma} and 
\[
\gamma_1=\int_{\R^N_+}m(z)w(z)^2z_N\,dz.
\]
\end{proposition}
\begin{proof} We write $r=r(\delta)$ and assume that $\delta$ and $r$ are small enough.

{\it Step 1.} We start from the gradient term. Using the natural change of variable, we obtain\[
\begin{split}
\int_\Omega|\nabla \varphi_\delta|^2&=\int_{D_2}\left|\nabla_x w_*(\Psi(x)/r) \right|^2\,dx=r^{-2}\int_{D_2}\left|w_*'(\Psi(x)/r)\frac{\Psi(x)}{|\Psi(x)|}D\Psi(x)\right|^2\,dx
\\
&=r^{-2}\int_{B^+_{2\NTr}}\left|w_*'(y/r)\frac{y}{|y|}D\Psi(\Phi(y))\right|^2\det D\Phi(y)\,dy.
\end{split}
\]
We now use~\eqref{eq:lemmaA1} and write 
\[
y= r z\qquad\text{ and }\qquad R=\frac\NTr{r}. 
\]
We obtain, calling from now on $\psi_{ij}=\psi_{ij}(0)$,
\begin{equation}\label{eq:passa1}
\int_\Omega|\nabla \varphi_\delta|^2=r^{N-2}\int_{B^+_{2R}}|w_*'(z)|^2\left(1+r z_N\left[2\sum_{i,j=1}^{N-1}\psi_{ij}\frac{z_iz_j}{|z|^2}-\alpha\right]+O(r^2|z|^2)\right)\,dz.
\end{equation}
By the exponential decay of $w$ and $w'$ (see Lemma \ref{le:decaypsialto}) we have, for all $z\in B_{2R}^+$, \[
%\begin{split}
|w_*'(z)|^2=\Big[w'(z)\zeta_R(|z|)+w(z)\zeta_R'(|z|)\Big]^2%\\&
\leq 2\Big[w'(z)^2+w(z)^2R^{-2}\Big]\leq Ce^{-\vartheta|z|},
%\end{split}
\]
for a suitable $\vartheta>0$. On the other hand, it is easy to check that there is a constant $C_0$, independent of $r$ (and thus also of $\delta$), such that \[
\left(1+rz_N\Big[2\sum_{i,j=1}^{N-1}\psi_{ij}\frac{z_iz_j}{|z|^2}-\alpha\Big]+O(r^2|z|^2)\right)\leq C_0,\qquad \text{for all }z\in B_{2R}^+.
\]
Hence, 
\[
%\begin{split}&
\int_{B^+_{2R}\setminus B_R^+}|w_*'(z)|^2\left(1+rz_N\Big[2\sum_{i,j=1}^{N-1}\psi_{ij}\frac{z_iz_j}{|z|^2}-\alpha\Big]+O(r^2|z|^2)\right)\,dz\leq C_0\int_{B^+_{2R}\setminus B_R^+}
e^{-\vartheta |z|}\,dz = O(r^2).
%\\
%&\leq CR^{N-1}e^{-\vartheta R}\leq Ce^{-\frac{\vartheta}{2}R}=Ce^{-\frac{\vartheta\NTr}{2r}},
%\end{split}
\]
On the other hand, in $B_R$ we have that that $w_*=w$ and
\[
\int_{B^+_{R}}|w_*'(z)|^2 O(r^2|z|^2)\,dz \le C r^2 \int_{\R^N_+} |w'(z)|^2|z|^2\,dz = O(r^2),
\]
again by exponential decay of $w$. Plugging this information into \eqref{eq:passa1} we have
\begin{equation}\label{eq:passa2}
\int_\Omega|\nabla \varphi_\delta|^2=r^{N-2}\int_{B^+_{R}}|w'(z)|^2\left(1+r z_N\left[2\sum_{i,j=1}^{N-1}\psi_{ij}\frac{z_iz_j}{|z|^2}-\alpha\right]\right)\,dz+O(r^N).
\end{equation}
To pass to integrals in the half-space, we notice that, by exponential decay, 
\[
\int_{\R^N_+}w'(z)^2|z|^2\,dz<+\infty
\qquad\implies\qquad
\int_{\R^N_+\setminus B_R^+}w'(z)^2(1+|z|)\,dz= O(r^N).
\]
In conclusion, we have proved that
\[
\int_\Omega|\nabla \varphi_\delta|^2=r^{N-2}\int_{\R^N_+}|w'(z)|^2\left(1+rz_N\Big[2\sum_{i,j=1}^{N-1}\psi_{ij}\frac{z_iz_j}{|z|^2}-\alpha\Big]\right)\,dz+O(r^N).
\]
We now observe that the radiality of $w$ entails \[
\int_{\R^N_+}w'(z)^2\frac{z_iz_j}{|z|^2}z_N\,dz=0,\qquad \text{if }i\not=j.
\]
Then one can compute, using also~\cite[Lemma~3.3]{nitakagi_cpam},\[
\sum_{i,j=1}^{N-1}\psi_{ij}\int_{\R^N_+}w'(z)^2\frac{z_iz_j}{|z|^2}z_N\,dz=\sum_{j=1}^{N-1}\psi_{jj}\int_{\R^N_+}\left|\frac{\partial w}{\partial z_j}\right|^2z_N\,dz=\alpha\gamma.
\]
We have thus concluded the first part of the statement (for $\delta$ and $r(\delta)$ small enough):\[
\begin{split}
\int_{\Omega}|\nabla \varphi_\delta|^2&=r^{N-2}\left(\int_{\R^N_+}|\nabla w(z)|^2+r\Big[2\alpha\gamma-(N+1)\alpha\gamma\Big]+O(r^2)\right)\\
&=r^{N-2}\left(\frac12\int_{\R^N}|\nabla w(z)|^2-r(N-1)\alpha\gamma+O(r^2)\right).
\end{split}
\]

{\it Step 2.} We now deal with the second part of the claim. With the same techniques as before (with the change of variables $y=\Psi(x)$ and $z=y/r$), we can compute\[
\begin{split}
\int_\Omega m_\delta \varphi_\delta^2\,dx&=\int_{D_2}m_\delta\varphi_\delta^2\,dx=\int_{B^+_{2\NTr}} m(y/r)w_*^2(y/r)\det D\Phi(y)\,dy\\
&=r^N\int_{B_{2R}^+}m(z)w_*^2(z)\left(1-\alpha z_Nr+O(r^2|z|^2)\right)\,dz,
\end{split}
\]
where $R=\NTr/r$ and we have used~\eqref{eq:lemmaA1}.

The exponential decay of $w$ allows to argue as in the previous step, yielding
\[
\begin{split}
r^{-N}\int_\Omega m_\delta \varphi_\delta^2\,dx&=\int_{\R^N_+}m(z)w^2(z)\Big(1-\alpha z_Nr\Big)\,dz+O(r^2)\\
&=\frac12\int_{\R^N}m(z)w^2(z)\,dz-r\alpha\int_{\R^N_+}m(z)w^2(z)z_N\,dz+O(r^2)\\
&=\frac12\int_{\R^N}m(z)w^2(z)\,dz-r\alpha\gamma_1+O(r^2),
\end{split}
\]
where $\gamma_1$ is defined in the statement.
\end{proof}
\begin{corollary}\label{cor:boundabove}
With the notation of Proposition~\ref{prop:boundabove}, for $\delta\to 0$ (and thus $r(\delta)\to 0$), we have
\begin{equation}\label{eq:boundabove1}
%\begin{split}
\od(\delta)\leq r^{-2}(\delta)\left\{\mu(B)+\frac{\alpha r(\delta)}{\int_{\R^N_+}mw^2}\left[\mu(B)\gamma_1-(N-1)\gamma\right]+o(r(\delta))\right\}.
%\end{split}
\end{equation}
\end{corollary}
\begin{proof}
We note that $\varphi_\delta$ is an admissible competitor for $\od(\delta)$, for $\delta$ and $r(\delta)$ small enough, thus \[
\od(\delta)\leq \frac{\int_\Omega|\nabla \varphi_\delta|^2}{\int_{\Omega}m_\delta\varphi_\delta}.
\]
Then it is enough to apply the expansions proved in Proposition~\ref{prop:boundabove}, also recalling the elementary expansion \[
\frac{a-c_1\eps+o(\eps)}{b-c_2\eps+o(\eps)}=
\left(\frac{a}{b}-\frac{c_1}{b}\eps+o(\eps)\right)\cdot\left(1+\frac{c_2}{b}\eps+o(\eps)\right)=\frac{a}{b}-\Big(\frac{c_1}{b}-\frac{a c_2}{b^2}\Big)\eps+o(\eps),
\]
with \[
a=\int_{\R^N_+}|\nabla w|^2,\quad b=\int_{\R^N_+}mw^2,\quad c_1=(N-1)\alpha\gamma,\quad c_2=\alpha \gamma_1\quad \text{and }\eps=r(\delta)\to0.\qedhere
\]
\end{proof}

In order to deduce the desired bound from above, we need a technical lemma.
\begin{lemma}\label{le:technical}
With the notation above, we have 
\begin{equation}\label{eq:technical2}
\mu(B)\gamma_1-(N-1)\gamma=-2\gamma+4\mu(B)\omega_N^{-\frac{N+1}{N}}\frac{\omega_{N-1}}{N(N+1)}\int_{\R^N_+}mw^2.
\end{equation}
\end{lemma}
\begin{proof}
We write $\mu=\mu(B)$.
First of all, we test the equation of $w$ in $\R^N_+$ with $z_N^2\partial_N w$:
\begin{equation}\label{eq:testaz^2}
\int_{\R^N_+}(-\Delta w)z_N^2\partial_N w\,dz=\mu\int_{\R^N_+}mwz_N^2\partial_N w\,dz.
\end{equation}
Using the divergence theorem, the decay to zero of $w$ at infinity and the relation
\[
\gamma=\frac1{N+1}\int_{\R^{N}_{+}}\left|\nabla w(z)\right|^{2}z_{N}\, dz=\frac12\int_{\R^N_+}\left(\frac{\partial w}{\partial z_N}\right)^2z_N\,dz.
\]
(see~\cite[Lemma~3.3]{nitakagi_cpam}) we obtain
\[
\begin{split}
&\int_{\R^N_+}(-\Delta w)z_N^2\partial_N w\,dz=\int_{\R^N_+}\nabla w\cdot \nabla (z_N^2\partial_N w)\,dz\\
&=2\int_{\R^N_+}z_N(\partial_N w)^2\,dz+\frac12\int_{\R^N_+}z_N^2\partial_N|\nabla w|^2\,dz=4\gamma- \int_{\R^N_+}z_N|\nabla w|^2\,dz=-(N-3)\gamma.
\end{split}
\]
On the other hand, let us denote the radius of $B$ as $\overline R=\omega_N^{-1/N}$. Using the divergence theorem, the definition of $m$ and $\gamma_1$, and the fact that $w$ is radial, we obtain
\begin{multline*}
\int_{\R^N_+}mwz_N^2\partial_Nw\,dz
=\frac{1}{2}\int_{B^+}z_N^2\partial_Nw^2\,dz-\frac{ \beta}{2}\int_{\R^N_+\setminus B^+}z_N^2\partial_Nw^2\,dz\\
=-\int_{\R^N_+}mw^2z_N\,dz+\frac{1+\beta}{2}\int_{\partial B^+}z_N^2 w^2\frac{z\cdot e_N}{|z|}\,d\mathcal H^{N-1}
=-\gamma_1+\frac{1+\beta}{N+1}\omega_{N-1}\overline R^{N+1}w^2(\overline R)
\end{multline*}
where we evaluated
\[
\int_{\partial B^+}z_N^2\frac{z\cdot e_N}{|z|}\,d\mathcal H^{N-1}=
\int_{B^+}\partial_N(z_N^2)\,dz
=\frac{2}{N+1}\omega_{N-1}\overline R^{N+1}.
\]
As a consequence, \eqref{eq:testaz^2} is equivalent to
\begin{equation}\label{eq:intermedia}
\mu\gamma_1-(N-1)\gamma=-2\gamma+\mu\frac{1+\beta}{N+1}\omega_{N-1}\overline R^{N+1}w^2(\overline R).
\end{equation}
To get rid of the dependence on $w(\overline R)$ we use the Pohozaev identity: testing the equation
with $\nabla w\cdot z$ we obtain\[
\int_{\R^N_+}(-\Delta w)\nabla w\cdot z\,dz=\mu\int_{\R^N_+}mw\nabla w\cdot z\,dz.
\]
On the left hand side, using the divergence theorem and the symmetry of $D^2w$, we get
\[
\int_{\R^N_+}(-\Delta w)\nabla w\cdot z\,dz=\int_{\R^N_+}|\nabla w|^2\,dz+\frac12\int_{\R^N_+}\nabla |\nabla w|^2\cdot z\,dz=\left(1-\frac{N}{2}\right)\int_{\R^N_+}|\nabla w|^2\,dz.
\]
On the other hand, for the right hand side, we use again the divergence theorem and the definition of $m$. Recalling that $\nabla w$ has a jump across $\partial B$, we obtain\[
\begin{split}
\int_{\R^N_+}mw\nabla w\cdot z\,dz
&=-\frac{N}{2}\int_{\R^N_+}mw^2\,dz+\frac{1+\beta}{2}\int_{\partial B^+}w^2(z)z\cdot \frac{z}{|z|}\,d\mathcal H^{N-1}\\
&
=-\frac{N}{2}\int_{\R^N_+}mw^2\,dz+\frac{1+\beta}{4}N\omega_N\overline R^{N}w^2(\overline R).
\end{split}
\]
All in all,
\[
\mu\int_{\R^N_+}mw^2\,dz=\int_{\R^N_+}|\nabla w|^2\,dz=\frac{\mu(1+\beta)}{4}N\omega_N\overline R^{N}w^2(\overline R).
\]
By plugging it into~\eqref{eq:intermedia} and recalling that $\overline R=\omega_N^{-1/N}$, we finally have the claim~\eqref{eq:technical2}.
\end{proof}

To conclude the proof, we need the following result.
\begin{lemma}[{\cite[Lemma~4.10]{mapeve}}]\label{le:MPV4.10}
Assume that, for positive constants $a,b,c,d$,
\[
\delta = a r^N \left( 1 - br + o(r)\right),\qquad
\nu = c r^{-2} \left( 1 - dr + o(r)\right),\qquad \text{as }r\to0^+.
\]
Then
\[
\nu = c a^{2/N}  \delta^{-2/N} \left( 1 -  \frac{a^{-1/N} (2b + Nd)}{N} \,\delta^{1/N} + o(\delta^{1/N})\right)\qquad \text{as }\delta\to0^+.
\]
\end{lemma}
\begin{proof}[Proof of Theorem~\ref{th:boundabove}]
Recalling Lemma \ref{le:rdelta}, Corollary \ref{cor:boundabove} and Lemma \ref{le:technical}, 
up to now we have obtained the following relations:\[
\begin{split}
\delta&=\frac12r^N\left(1-\frac{2}{N+1}\omega_{N-1}\omega_N^{-\frac{N+1}{N}}\alpha\,r+o(r)\right),\\
\frac{\int_\Omega|\nabla \varphi_\delta|^2}{\int_\Omega m_\delta\varphi_\delta^2} &=\frac{\mu(B)}{r^2}\left\{1-\alpha r\left[\frac{2\gamma}{\mu(B)\int_{\R^N_+}mw^2}-4\omega_N^{-\frac{N+1}{N}}\frac{\omega_{N-1}}{N(N+1)}\right]+o(r)\right\}.
\end{split}
\]
To merge them together and deduce the claim, we apply Lemma~\ref{le:MPV4.10} with the 
obvious choice of the parameters.
In particular we obtain
\[
\begin{split}
\frac{2b}{N}+d&=\frac{4}{N(N+1)}\omega_{N-1}\omega_N^{-\frac{N+1}{N}}\alpha+\alpha \left[\frac{2\gamma}{\int_{\R^N_+}|\nabla w|^2}-4\omega_N^{-\frac{N+1}{N}}\frac{\omega_{N-1}}{N(N+1)}\right]\\
&=\frac{2\alpha\gamma}{\int_{\R^N_+}|\nabla w|^2}.
\end{split}
\]
Since $\varphi_\delta$ is an admissible competitor for $\od(\delta)$, we obtain\[
\od(\delta)\leq \frac{\int_\Omega|\nabla \varphi_\delta|^2}{\int_\Omega m_\delta\varphi_\delta^2}=2^{-2/N}\mu(B)\delta^{-2/N}\left(1-2^{1/N}\frac{2\alpha\gamma}{\int_{\R^N_+}|\nabla w|^2}\delta^{1/N}+o(\delta^{1/N})\right).\qedhere
\]
\end{proof}

\section{Blow-up argument}\label{sec:blowup}

This section is mainly devoted to the proof of Theorem \ref{th:intro_D}. We follow 
the approach introduced in~\cite[Section 4]{nitakagi_cpam}, based on a blow-up analysis. Differently from \cite{nitakagi_cpam}, we cannot deal directly with local maximum points, and we are forced to consider only global maximizers, to avoid the vanishing of the blow-up sequence.

Let $\od(\delta)$ (introduced in Definition \ref{def:od}) be achieved by an open set 
$D_{\delta}$, and let $u_\delta\in H^1(\Omega)$ be the $L^2$-normalized, positive principal 
eigenfunction, having global maximum at $P_\delta\in\overline{\Omega}$:
\[
\od(\delta) = \la(D_{\delta}), \qquad \int_{\Omega}u_\delta^2=1,\qquad 
u(P_\delta)=\max_{\overline{\Omega}} u.
\]
First we  show that $\dist(P_{\delta},\partial 
\Omega)=o(\delta^{1/N})$ as $\delta\to0$ (see Lemma \ref{le:step1}, \ref{le:step2}); this allows to prove that
\[
\liminf_{\delta \to 0} \od(\delta)\cdot \delta^{-2/N} \geq 2^{-2/N}I_{\Mcal}=I_{\Mcal_2}
\]
which, together with Remark \ref{rem:M1M2}, yields Theorem \ref{thm:bound_intro}. 
At the same time, this first blow-up procedure allows to obtain a strong non-vanishing property (Proposition \ref{prop:nonvanish}), which in turn allows to deal with local maximizers (Lemma
\ref{lem:massimivicini}). 
Then we 
show that $P_{\delta}$ actually belongs to 
$\partial \Omega$ and it is unique.
As a consequence the qualitative properties of $D_{\delta}$ stated in Theorem~\ref{th:intro_D} 
follow. 

We first show that $\dist(P_\delta,\partial \Omega)=O(\delta^{1/N})$.
\begin{lemma}\label{le:step1}
There exists $C>0$ such that, for $\delta$ small enough,
\[
\dist(P_\delta,\partial \Omega)\leq C\delta^{1/N}.
\]
\end{lemma}
\begin{proof}
We argue by contradiction assuming that there is a sequence $\delta_j\to 0$ such that 
\begin{equation}\label{eq:assurdo}
\rho_j:=\frac{\dist(P_j,\partial \Omega)}{\delta_j^{1/N}}\to+\infty,\qquad \text{as }j\to+\infty,
\end{equation}
with $P_j:=P_{\delta_j}$. We prove the lemma by a blow-up procedure, in several steps.

\emph{Step 1: convergence of the blow-up sequence.} Introducing the rescaled sets and functions
\begin{equation}\label{eq:defscalings}
\Omega_j:=\frac{\Omega-P_j}{\delta_j^{1/N}},\quad D_j:=\frac{D_{\delta_j}-P_j}{\delta_j^{1/N}},\quad m_j:=\ind{D_j}-\beta\ind{\Om_\delta\setminus D_j},\quad v_j(z):=\delta_j^{1/2}u_{\delta_j}(P_j+\delta_j^{1/N}z),
\end{equation}
 it is easy to check that 
 \[
|\Omega_j|=\frac{|\Omega|}{\delta_j},\quad |D_j|=1,\quad
\begin{cases}
-\Delta v_j=\lambda_j m_j v_j &\text{ in }\Om_j\\
\partial_\nu v_j=0 &\text{ on }\partial \Omega_j ,\\
\end{cases}
\qquad \lambda_j=\frac{\int_{\Om_j}|\nabla v_j|^2}{\int_{\Om_j}m_j v_j^2}
=\delta_j^{2/N}\od(\delta_j).
\]
By definition,   $\rho_j=\dist(0,\partial \Omega_j)$, so that the ball $B_{\rho_j}$ centered at the origin is contained in $\Omega_j$ for all $j$. Moreover, by definition of $m_j$ and Remark~\ref{rem:M1M2} we infer that, up to (not relabeled) subsequences, 
\begin{equation}\label{eq:mconv}
m_j\rightharpoonup m,\quad \text{weakly $*$ in $L^\infty_{\loc}(\R^N)$},\qquad
\od(\delta_j)\delta_{j}^{2/N}=\la_j\to \lambda\in[0,I_{\Mcal_2}],\quad
\text{as $j\to+\infty$.}
\end{equation}
On the other hand, thanks to the normalization of $u_\delta$, we have 
\begin{equation}\label{eq:normaL2}
\int_{\Omega_j}v^2_j(z)\,dz=\int_{\Omega}u^2_{\delta_j}(x)\,dx=1.
\end{equation}
Let us now fix $B_r=B_r(0)$: we observe that, for all $j$ sufficiently big, it holds $B_r\subset B_{\rho_j}\subset \Omega_j$, as $\rho_j\to+\infty$ by \eqref{eq:assurdo}.
Then, we can  compute 
\begin{equation}\label{eq:boundvj}
\int_{B_r}|\nabla v_j|^2\leq \int_{\Omega_j}|\nabla v_j|^2=\int_{\Omega_j}\lambda_j m_j v_j^2\leq C \int_{\Omega_j}v_j^2=C,
\end{equation}
where we used \eqref{eq:mconv}, and the fact that $\|m_j\|_{L^\infty(\Omega_{j})}\leq 1$.

As a consequence of~\eqref{eq:boundvj} and of~\eqref{eq:normaL2}, we have a bound on the $H^1(B_r)$ norm of $v_j$, uniform in $r$. Thus, there exists $v\in H^{1}(\R^{N})$ such  that
\[
v_j\rightharpoonup v,\qquad \text{weakly in $H^1_{\rm loc}(\R^{N})$, strongly in $L^p_{\rm loc}(\R^{N})$, for every $p\in \left[1,2^*\right)$, 
as $j\to+\infty$}. 
\]
where $2^*$ is the usual Sobolev exponent, $v\ge0$ a.e. in $\R^N$, and
 \begin{equation}\label{eq:stimalamv}
\|\lambda_j m_j v_j\|_{L^p(B_r)}\leq \la_j\|m_j\|_{L^\infty(B_r)}\|v_j\|_{L^p(B_r)}\leq C\|v_j\|_{L^p(B_r)}.
\end{equation}
Classical elliptic estimates, (see e.g.\ ~\cite[Theorem~9.11]{gilbargtrudinger}) yield 
\[
\|v_j\|_{W^{2,p}(R')}\leq C \Big(\|v_j\|_{L^p(B_r)}+\|\lambda_j m_j v_j\|_{L^p(B_r)}\Big),\qquad \text{for all }2\leq p< 2^*\text{ and }\overline{R'}\subset B_r,
\] 
where $C$ depends only on $N, p, B_r, R'$. Then \eqref{eq:stimalamv}, the Sobolev embedding 
and~\eqref{eq:boundvj} yield
\[
\|v_j\|_{W^{2,p}(R')}\leq C \|v_j\|_{L^p(B_r)}\leq C\|v_j\|_{H^1(B_r)}\leq C,
\quad \text{for all }2\leq p< 2^*\text{ and }\overline{R'}\subset B_r,
\]
where $C$ again depends on $N,p,B_r, R'$.

Now, if $N=2,3$ there exists a $p\in(N,2^*)$; then Morrey's Theorem implies that $W^{2,p}(R')$ is 
continuously embedded in $C^{1,\alpha}(R')$, $\alpha= 1-N/p$. 
By Ascoli-Arzel\`a's Theorem one deduces that the sequence $(v_j)_j$ is precompact in 
$C^{1,\vartheta}(R')$ for all $\vartheta<\alpha$.
Thus, up to subsequences, 
\[
v_j\to v \qquad \text{in $C^{1,\vartheta}_{\rm loc}(\R^N)\cap H^1_{\rm loc}(\R^N)$}, 
\qquad
v\in H^1(\R^N)
\]
as $j \to +\infty$, by arbitrarity of $B_r$ and $R'$ and by the uniform bounds obtained above. 
On the other hand, if $N\ge 4$, then we can use a bootstrap argument to prove in a finite number of steps 
the same uniform bound as above. 

Finally, let $\varphi\in C^\infty_c(\R^N)$ and $j$ sufficiently large so that 
$K:=\supp(\varphi)\subset B_{\rho_j}\subset \Omega_j$. Then 
\[
\int_K \nabla v_j\cdot \nabla \varphi=\lambda_j\int_K m_j v_j\varphi,
\]
and taking into account \eqref{eq:mconv} we obtain that the limit function $v$ is a weak solution to 
\begin{equation}\label{eq:eqlimite}
-\Delta v= \lambda m v,\qquad \text{in }\R^N. 
\end{equation}

\emph{Step 2: the blow-up limit $v$ is non-trivial and $\lambda>0$.} 
By the $L^2$ normalization, 
\[
0<\int_{\Omega_j} m_jv_j^{2}=
\int_{D_j} v_j^{2}-\beta \int_{\Omega_j\setminus D_j} v_j^{2}=
\int_{D_j} v_j^{2}-\beta \left(1-\int_{D_j} v_j^{2}\right) 
=(\beta+1)\int_{D_j} v_j^{2}-\beta.
\] 
Since $P_j$ is a global maximum point we infer
\begin{equation}\label{eq:intpositivo}
0<\frac{\beta}{\beta+1} < \int_{D_j}v_j^2 \le |D_j| \sup_{D_j}v_j^2 = v_j^2(0)\to v^2(0)
\end{equation}
by pointwise convergence, thus $v$ is non-trivial. Testing \eqref{eq:eqlimite} with $v$ 
we obtain that $\lambda>0$. 

\emph{Step 3: the limit weight $m$ is admissible for $I_{\Mcal}$ (see \eqref{eq:defMk}).} 
First, we take $\vfi=\ind{E}$
with $E:=\{x\in \R^{N} : m< -\beta\}\cap B_{r}(0)$ for $r>0$; then, by \eqref{eq:mconv},
\[
0\leq \int_{\Omega_j}(m_j+\beta)\ind{E}\to \int_{E}
m+\beta \leq 0,
\]
showing that ${m}\geq -\beta$; in an analogous way it is possible to show that 
$m \leq 1$.
In addition, taking as test function $v$ in \eqref{eq:eqlimite} we obtain
\[
0<\int_{\R^{N}}|\nabla v|^2=\lambda \int_{\R^{N}}mv^2,
\]
showing that $|\{x\in \R^{N} :  m>0\}|>0$. Finally,  for all $r>0$ 
\[
\int_{B_r}m+\beta =\lim_{j\to+\infty} \int_{\R^{N}}(m_j+\beta)\ind{B_{r}}\leq (1+\beta)|D_j|=
 1+\beta,
\]
hence, taking $r\to +\infty$ and using monotone convergence, 
\begin{equation*}%\label{eq:mMk}
\int_{\R^{N}}m+\beta\leq (1+\beta).
\end{equation*}

\emph{Step 4: conclusion.} By the previous steps and Remark \ref{rmk:scaling} we have
\[
\lambda=\frac{\int_{\R^N}|\nabla v|^2}{\int_{\R^N}m v^2}\geq I_{{\mathcal M}}=
2^{\frac{2}{N}}I_{{\mathcal M}_{2}},
\]
which contradicts \eqref{eq:mconv}.
\end{proof}
\begin{remark}\label{rem:maxglobvsloc}
We notice that Step 2 is  the only point in the proof of Lemma \ref{le:step1} in which we need 
that $P_\delta$ is a \emph{global} maximum point. Consequently, the blow-up argument works also for local maxima, up to the fact that, at this point, we can not exclude that the blow-up limit is trivial.
\end{remark}

Next we improve the previous result, showing that $\dist(P_\delta,\partial \Omega)=o(\delta^{1/N})$.

\begin{lemma}\label{le:step2}
Let $P_\delta$ be a global maximum point for $u_\delta$. Then
\[
\limsup_{\delta\to 0}\frac{\dist(P_\delta,\partial \Omega)}{\delta^{1/N}}=0.
\]
\end{lemma}

\begin{proof}
Again, we use a blow up argument. Since here we work near the boundary 
$\partial \Omega$, we exploit the diffeomorphism introduced 
in Section \ref{sec:bfa}. We subdivide the proof into two steps.

\emph{Step 1: blow-up procedure.} Let us consider any sequence $P_k=P_{\delta_k}$ of global 
maximum points for $u_{\delta_k}$, $\delta_k\to0$. W.l.o.g. we can assume that $P_k\in\Omega$, 
for all $k$. Then, up to subsequences, thanks to Lemma~\ref{le:step1}, we have that 
$P_k\to P\in\partial \Omega$, as $k\to+\infty$.  

As in Section \ref{sec:bfa}, we assume that $P=0$, that the exterior normal to $\partial\Omega$ at $0$ 
coincides with $-e_N$,  and we apply the diffeomorphism $\Psi=\Phi^{-1}$ (see \eqref{eq:defPsi}), writing $Q_k=\Psi(P_k)\in B^+_\NTr$.
The transformed eigenfunction is defined by 
\begin{equation}\label{eq:vk}
v_k(y):=u_{\delta_k}(\Phi(y)),\qquad y\in  \overline B^{+}_{2\NTr}.
\end{equation}
It is then easy to extend the function $v_k$ by symmetry in the whole $B_{2\NTr}$  by defining
\begin{equation}\label{eq:vtildek}
\widetilde v_k(y):=
\begin{cases}
v_k(y),\qquad &\text{if }y_N\geq 0,\\
v_k(y',-y_N),\qquad &\text{if }y_N<0.
\end{cases}
\end{equation}
At this point, we can introduce the blow-up sequences, for $\delta_k>0$,
\begin{equation}\label{eq:wk}
\Omega_k=\frac{B_{2\NTr}-Q_k}{\delta_k^{1/N}},\qquad w_k(z)=\delta_k^{1/2}\,\widetilde v_k\left(Q_k+\delta_k^{1/N}z\right),\qquad z\in \overline{B_{\NTr\delta_k^{-1/N}}}.
\end{equation}
Let us denote 
\begin{equation}\label{eq:alphak}
Q_k=(q'_k,\alpha_k\delta_k^{1/N}), \quad \text{with $ q_k'\in \R^{N-1}$, $\alpha_k>0$ and   $Q_{k}\to 0$.}
\end{equation}
By Lemma~\ref{le:step1} we deduce that the sequence $(\alpha_k)_k$ is  bounded, so that $\alpha_k\to\bar\alpha\ge0$. Then, to prove the lemma we have to show that $\bar \alpha= 0$.
We notice that 
\begin{equation}\label{eq:derivataNnulla}
\frac{\partial \widetilde v_k}{\partial y_N}=0 \text{ on }\{y_N=0\},
\qquad\text{i.e.\ }
\qquad
\frac{\partial w_k}{\partial z_N}=0 \text{ on }\{z_N=-\alpha_k\},
\end{equation}
therefore \[
w_k\in C^{1,\vartheta} \Big(\overline B_{\NTr\delta_k^{-1/N}}\setminus \{z_N=-\alpha_k\}\Big)\cap C^1\Big(\overline B_{\NTr\delta_k^{-1/N}}\Big)\cap H^1\Big(B_{\NTr\delta_k^{-1/N}}\Big),
\]
thanks to the smoothness of the diffeomorphism and to the regularity of $u_{\delta_k}$. 
Moreover the $L^2(\Omega)$ normalization of $u_{\delta_k}$ entails 
 \[
\int_{B_{\NTr\delta_k^{-1/N}}}w_k^2(z)\,dz\leq C,
\]
for some constant $C>0$ independent of $k$, and
$w_k$ solves, for almost all $z\in B_{\NTr \delta_k^{-1/N}}\setminus \{z_N=-\alpha_k\}$, the elliptic equation 
\begin{equation}\label{eq:lapltrasformato}
-\sum_{i,j=1}^N a_{ij}^k(z)\frac{\partial^2 w_k}{\partial z_i\partial z_j}(z)-\delta_k^{\frac1N}\sum_{j=1}^Nb_j^k(z)\frac{\partial w_k}{\partial z_j}(z)=\delta_k^{\frac2N}\od(\delta_k)m_k(z)w_k(z),
\end{equation}
where we denote, for $z\in B_{\NTr \delta_k^{-1/N}}$,
\begin{equation}\label{eq:mk}
m_k(z)=
\begin{cases}
m_{\delta_k}\Big(\Phi(Q_k+\delta_k^{1/N}z)\Big),\qquad &\text{if }z_N\geq -\alpha_k,\\
m_{\delta_k}\Big(\Phi(q_k'+\delta_k^{1/N}z',-\delta_k^{1/N}(\alpha_k+z_N)\Big),\qquad &\text{if }z_N< -\alpha_k.
\end{cases}
\end{equation}
The coefficients $a_{ij}^k$ and $b_j^k$ are defined as follows. First of all, let 
\[
a_{ij}(y)=\sum_{l=1}^N\frac{\partial \Psi_i}{\partial x_l}(\Phi(y))\frac{\partial \Psi_j}{\partial x_l}(\Phi(y)),\quad 
b_j(y)=(\Delta \Psi_j)(\Phi(y)),
\qquad1\leq i,j\leq N.
%\begin{split}
%a_{ij}(y)&=\sum_{l=1}^N\frac{\partial \Psi_i}{\partial x_l}(\Phi(y))\frac{\partial \Psi_j}{\partial x_l}(\Phi(y)),\qquad 1\leq i,j\leq N,\\
%b_j(y)&=(\Delta \Psi_j)(\Phi(y)),\qquad j=1,\dots, N.
%\end{split}
\]
Then we set \[
\begin{split}
a_{ij}^k(z)&=
\begin{cases}
a_{ij}(Q_k+\delta_k^{1/N}z), &\text{ if }z_N\geq -\alpha_k,\\
(-1)^{\delta_{jN}+\delta_{iN}}a_{ij}\Big(q'_k+\delta_k^{1/N}z',-(\alpha_k+z_N)\delta_k^{1/N}\Big), &\text{ if }z_N<-\alpha_k,
\end{cases}\\
b_j^k(z)&=
\begin{cases}
b_j(Q_k+\delta_k^{1/N}z), &\text{ if }z_N\geq -\alpha_k,\\
(-1)^{\delta_jN}b_j\Big(q'_k+\delta_k^{1/N}z',-(\alpha_k+z_N)\delta_k^{1/N}\Big), &\text{ if }z_N<-\alpha_k,
\end{cases}
\end{split}
\]
where $\delta_{nm}$ is the Kronecker delta. With elliptic regularity arguments as in Lemma \ref{le:step1}, we deduce that \[
w_k\to w_\infty,\qquad \text{in }C^{1,\vartheta}_\loc(\R^N) \cap W^{2,p}_\loc(\R^N),
\]
for $\vartheta<1$, $p<+\infty$. Indeed, all the coefficients $a_{ij}^k,b_j^k$ are Lipschitz continuous with uniform constants with respect to $k$ in $B_{\NTr\delta_k^{-1/N}}$, except $b_N^k$, but the product $b_N^k(z)\partial_{z_N}w_k(z)$ is Lipschitz continuous uniformly in $k$, thanks to~\eqref{eq:derivataNnulla}, and it can be treated as a right hand side term to obtain 
the Schauder estimates. Moreover, $w_\infty\in C^{1,\vartheta}(\R^N)\cap W^{2,p}(\R^N)$ is 
non-trivial (see Step 2 in the proof of Lemma \ref{le:step1}; again, this is the only part of the proof in which we need $P_\delta$ to be a \emph{global} maximum point).

Now we recall that, as $k\to +\infty$, $Q_k\to 0$, $\delta_k^{2/N}\od(\delta_k)
\to\lambda\in(0,I_{\Mcal_2}]$, $m_{k}\stackrel{*}{\rightharpoonup} m$  weakly $*$ in 
$L^\infty_\loc(\R^N)$; moreover, $A^k\to D\Psi(0)=[D\Phi(0)]^{-1}=\Id$ and $\delta^{1/N}b^k\to0$. 
We infer that  $w_\infty$ satisfies
\begin{equation}\label{eq:lim_2}
\begin{cases}
-\Delta w_\infty =\lambda m w_\infty & \text{in } \R^N,\\
\max_{\R^N} w_\infty = w_\infty(0)\\
w_\infty(z',-z_N - \bar\alpha) = w_\infty(z',z_N+\bar\alpha),
\end{cases}
\end{equation}
because the maximum of $w_k$ is at $z=0$, and $\alpha_k\to\bar \alpha$.

\emph{Step 2: analysis of the limit problem.}
First of all, let us show that $m\in \mathcal M_2$ (see \eqref{eq:defMk}). 
\\
As in the  proof of Lemma \ref{le:step1}, from the properties of the weak $*$ $L^\infty_\loc(\R^N)$ convergence, one deduces that $-\beta\leq m(z)\leq 1$ for a.e. $z\in \R^N$ and that $|\{x\in \R^{N} : m>0\}|>0$.
On the other hand, it is more delicate to check that \[
\int_{\R^N} (m+\beta)\leq 2(1+\beta).
\]
First of all, we introduce: 
\begin{equation}\label{eq:Dk}
%\begin{split}
D_{\delta_k}=\{x\in \Omega : m_{\delta_k}(x)=1\},\qquad |D_{\delta_k}|=\delta_k,\qquad
\widetilde D_k:=\widetilde D_k^+ \cup \widetilde D_k^-,
%\end{split}
\end{equation}
where
\[
\begin{split}
\widetilde D_k^+&=\left\{z\in B_{\NTr\delta_k^{-1/N}} : z_N\ge -\alpha_k,\ (z',z_N)\in \frac{\Psi(D_{\delta_k})-Q_k}{\delta_k^{1/N}}\right\},\\
\widetilde D_k^-&=\left\{z\in B_{\NTr\delta_k^{-1/N}} : z_N\le -\alpha_k,\ (z',-z_N)\in \frac{\Psi(D_{\delta_k})-(q'_k,-\alpha_k\delta_k^{1/N})}{\delta_k^{1/N}}\right\}.
\end{split}
\]
Let us fix a ball $B_r$ centered at the origin and of radius $r>0$, and $\delta_k$ small. We need to estimate the measure of the set ${\widetilde D_k}\cap B_r$. To this aim we notice that, by reflection,
\[
|\widetilde D_k^-\cap B_r(0)| = |\widetilde D_k^+\cap B_r (0',-2\alpha_k)|\le |\widetilde D_k^+\cap B_r (0)|:
\]
this because $\widetilde D_k^+ \subset \{z_N\ge-\alpha_k\}$ and $B_r (0',-2\alpha_k)\cap \{z_N\ge-\alpha_k\} \subset B_r (0)$. On the other hand, we have \[
\delta_k=\int_{\Omega}\ind{D_{\delta_k}}(x)\,dx\geq \int_{B_r\cap\{z_N\ge-\alpha_k\}}\delta_k\ind{D_{\delta_k}}\Big(\Phi(Q_k+\delta_k^{1/N}z)\Big)\,\det\Big(D\Phi(Q_k+\delta_k^{1/N}z)\Big)\,dz,
\]
by making the change of variable $z=\frac{\Psi(x)-Q_k}{\delta_k^{1/N}}$.
Then, by definition of $\widetilde D_k^+$ and using~\eqref{eq:lemmaA1},\[
\delta_k\geq\int_{B_r}\ind{{\widetilde D_k^+}}(z)\delta_k\Big(1-\alpha\delta_k^{1/N}(\alpha_k+z_N)+O\Big(|Q_k+\delta_k^{1/N}z|^2\Big)\Big)\,dz,
\]
where $\alpha_k$ is defined in 
\eqref{eq:alphak} (and $\alpha$ in 
\eqref{eq:alfagamma}). 
All in all, we obtain 
\begin{equation}\label{eq:stimaaltoDk}
|{\widetilde D_k}\cap B_r|\leq 2|{\widetilde D_k^+}\cap B_r| \leq 2\Big[1+C_1(r)\delta_k^{1/N}+C_2(r)\Big(|Q_k|^2+\delta_k^{2/N}|\Big)\Big]\le2+o(1),
\end{equation}
as $k\to+\infty$, since \[
C_1(r)=\alpha\int_{{\widetilde D_k}\cap B_r}(\alpha_k+|z_N|)\,dz\leq \alpha\int_{B_r}(\alpha_k+|z_N|)\,dz
\]
is uniformly bounded with respect to $k$, as the same is true for the sequence $(\alpha_k)$,
and $Q_k\to 0$.

At this point, exploiting  the weak $*$ $L^\infty(B_r)$ convergence of $m_k$ to $m$, we have 
\[
\begin{split}
\int_{B_r}(m+\beta)&=\lim_{k\to+\infty}\int_{B_r}(m_k+\beta)=\lim_{k\to+\infty}\int_{\R^N}(m_k+\beta)\ind{B_r}\\
&\leq \lim_{k\to+\infty}(1+\beta)|{\widetilde D_k}\cap B_r|\leq \lim_{k\to+\infty}2(1+\beta)\left[1+o_k(1)\right]=2(1+\beta).
\end{split}
\]
Then, by monotone convergence, we conclude\[
\int_{\R^N}(m+\beta)=\lim_{r\to+\infty}\int_{B_r}(m+\beta)\leq 2(1+\beta).
\]
Once $m\in \mathcal M_2$, we can argue as in Lemma \ref{le:step1}, showing that $\int_{\R^N}mw_\infty^2>0$ and finally, using 
Theorem \ref{th:boundabove},
\[
I_{\Mcal_2}\geq \lambda=\frac{\int_{\R^N}|\nabla w_\infty|^2}{\int_{\R^N}mw_\infty^2}\geq I_{\Mcal_2}.
\]
As a consequence, denoting with $w_{[2]}$ a (positive) eigenfunction associated to $I_{\Mcal_2}$, we have that $w_\infty(z)=c w_{[2]}(z+z_0)$ for some $c>0$ and $z_0\in \R^N$. Recalling that $w_{[2]}$ is radially symmetric, with a unique maximum point at $0$, we deduce that
$w_\infty$ is symmetric with respect to an hyperplane if and only if such hyperplane passes through its maximum point. Recalling \eqref{eq:lim_2} we obtain that  $\bar \alpha=0$ and the proof is concluded.
\end{proof}
\begin{remark}\label{re:limpro}
Note that in the  proof of Lemma \ref{le:step2} we also show  that, given any sequence $\delta_k\to0$, then some non-relabelled subsequences $(w_{k})_k$ and $(m_{k})_k$, as introduced in
\eqref{eq:wk},  \eqref{eq:mk} satisfy
\[
w_{k}\to w_{\infty},\quad \text{in } C^{1,\vartheta}_\loc(\R^N)\cap W^{2,p}_\loc(\R^N),\quad m_{k}\stackrel{*}{\rightharpoonup} m \quad \text{weakly *  in } L_{\loc}^{\infty}(\R^{N})
\]
with 
\[
\qquad w_{\infty}=cw_{[2]},\qquad
m=\ind{2^{1/N}B}-\beta\ind{2^{1/N}B^c},\]
where $c>0$ and $w_{[2]}$ 
is the principal eigenfunction
associated with $I_{\Mcal_2}$ 
(see Theorem  \ref{thm:limit_pb} and Remark \ref{rmk:scaling}).

Finally, by the bathtub 
principle applied to $u_\delta$ and the reflection argument, we have that there exist 
$t_k>0$ such that 
\begin{equation}\label{eq:dksopliv}
\widetilde D_k = \{z\in B_{\NTr\delta_k^{-1/N}} : m_k(z)=1\}
=
\{z : w_k(z)> t_k\},
\qquad
\partial \widetilde D_k=\{z : w_k(z)= t_k\}.
\end{equation}
\end{remark}

In the next lemma, we draw some consequences of equation \eqref{eq:stimaaltoDk} on $\widetilde D_k$ and $t_k$.
\begin{lemma}\label{le:improvedconv}
Under the previous notation:
\begin{enumerate}
\item Let $A\supseteq 2^{1/N}B$ be a bounded set. We have that
\begin{equation}\label{eq:limmis}
\lim_{k\to+\infty}|{\widetilde D_k}\cap A|=2.
\end{equation}
\item $\displaystyle \lim_k  t_k = 
 t_\infty:= \inf_{2^{1/N}B}w_\infty>0$.
\item For all $\eps\in (0,1)$, 
if $k$ is sufficiently large then there exists a connected component 
$\widetilde D_k^{\text{in}}$ of $\widetilde D_k$ satisfying
\begin{equation}\label{eq:BdentroDk}
2^{1/N}(1-\eps)B\subset {\widetilde D_k^{\text{in}}}\subset2^{1/N}(1+\eps)B.
\end{equation}
\end{enumerate}
\end{lemma}
\begin{proof} 
We keep the same notation of Lemma \ref{le:step2} and Remark \ref{re:limpro}.

\emph{Step 1: proof of \eqref{eq:limmis}.}
Let $A$ be as in the statement. Since $A\subset B_r$, for some $r$, 
\eqref{eq:stimaaltoDk} yields
\begin{equation}\label{eq:boundlim}
\limsup_{k\to+\infty}|{\widetilde D_k}\cap A|\leq 2.
\end{equation}
Taking into account Remark \ref{re:limpro}, the lower semicontinuity of the norm with respect to the weak convergence in $L^{2}(A)$ and \eqref{eq:boundlim}, we obtain
\[
2(1+\beta)^2=\int_{A}(m+\beta)^2\leq \liminf_{k\to+\infty}\int_{A}(m_k+\beta)^2\le (1+\beta)^2\limsup_{k\to+\infty}|{\widetilde D_k}\cap A|\leq 2(1+\beta)^2,
\]
yielding \eqref{eq:limmis}.

\emph{Step 2: $\liminf_k t_k \ge  t_\infty$.} We claim that, for every 
$\eps>0$, 
\[
\partial \widetilde D_k \cap 2^{1/N}(1+\eps)B \neq\emptyset,
\]
for $k$ large. Indeed, on the one hand $0\in \widetilde D_k 
\cap 2^{1/N}(1+\eps)B$. On the other hand, recalling \eqref{eq:stimaaltoDk}
\[
|\widetilde D_k \cap 2^{1/N}(1+\eps)B| \le 2+o(1) < 
2(1+\eps)^N=|2^{1/N}(1+\eps)B|, 
\]
for $k$ large, whence 
$2^{1/N}(1+\eps)B \setminus \widetilde D_k\neq \emptyset$ and the claim follows.

Then
\[
 t_k \equiv \left.w_k\right|_{\partial \widetilde D_k} \ge 
\inf_{2^{1/N}(1+\eps)B} w_k \ge \inf_{2^{1/N}(1+\eps)B}w_\infty + o(1),
\]
for $k$ large. Since $\eps$ is arbitrary, step 2 follows.

\emph{Step 3: $\limsup_k t_k \le  t_\infty$.} We choose $A=2^{1/N}B$ in \eqref{eq:limmis} and we observe that, by uniform convergence to the decreasing function $w_\infty$, for every  
$\sigma>0$ small there exists $0<\eps<1$ such that
\begin{equation}\label{eq:unif_conv}
\sup_{A\setminus 2^{1/N}(1-\eps)B} w_k \le \inf_{2^{1/N}B}
w_\infty + \sigma = t_\infty + \sigma.
\end{equation}
Let us assume by contradiction that, up to subsequences
\[
 t_k \to  t_\infty + \sigma' 
\]
for some $\sigma'>0$. Choosing 
$\sigma<\sigma'/2$ in \eqref{eq:unif_conv} we obtain, for $k$ large,
\[
\sup_{A\setminus 2^{1/N}(1-\eps)B} w_k \le  t_k - \sigma
\qquad\implies\qquad
A\cap \widetilde D_k \subset 2^{1/N}(1-\eps)B.
\]
Then we find a contradiction with \eqref{eq:limmis}, as
\[
2 = \lim_k |{\widetilde D_k}\cap A| \le |2^{1/N}(1-\eps)B| = 2(1-\eps)^N.
\]

\emph{Step 4: proof of \eqref{eq:BdentroDk}.}
Again by uniform convergence, for every $0<\eps<1$ there exists $\sigma>0$ such that 
\[
\inf_{2^{1/N}(1-\eps)B} w_k \ge \inf_{2^{1/N}B}
w_\infty + \sigma =  t_\infty + \sigma.
\]
By the second conclusion we 
deduce that, for $k$ large
\[
\inf_{2^{1/N}(1-\eps)B} w_k \ge  t_k + \frac{\sigma}{2}
\qquad\implies\qquad
 2^{1/N}(1-\eps)B
\subset \widetilde D_k.
\]
In the same way, there exists $\sigma>0$ such that 
\[
\sup_{2^{1/N}(1+\eps)\partial B} w_k \le \inf_{2^{1/N}B}
w_\infty - \sigma =  t_\infty - \sigma\le t_k - \frac{\sigma}{2},
\]
for $k$ large, so that 
\[
\widetilde D_k \cap  2^{1/N}(1+\eps)\partial B =\emptyset.
\qedhere
\]
\end{proof}
\begin{remark}\label{rem:C2}
By \eqref{eq:BdentroDk} we have that, for $k$ sufficiently large, $m_k(z)=1$ for all $z\in 2^{1/N}(1-\eps)B$, thus 
\begin{equation}\label{eq:convC2}
w_k\to w_\infty,\qquad \text{in }C^2(2^{1/N}(1-\eps)B),
\end{equation}
and \[
w_\infty\in C^2(2^{1/N}(1-\eps)B)\cap C^{1,\vartheta}_\loc(\R^N).
\]
This follows by elliptic regularity, recalling that $w_{k}$ satisfies \eqref{eq:lapltrasformato}.
\end{remark}

The results of Lemma \ref{le:improvedconv} can be easily translated in terms of $u_\delta$, $D_\delta$.
\begin{proposition}\label{prop:nonvanish}
Let $D_\delta$ and $u_\delta$ (as above) achieve $\od(\delta)$, and $P_\delta$ be a global maximum point for $u_\delta$.
\begin{enumerate}
\item There exists a universal constant $\sigma>0$ such that, for $\delta$ small enough,
\[
\inf_{D_\delta} u_\delta \ge \sigma \delta^{-1/2}.
\]
\item For every $\eps\in (0,1)$, 
if $\delta$ is sufficiently small then there exists a connected component $D_\delta^{\text{in}}$ 
of $D_\delta$ satisfying
\begin{equation}\label{eq:BdentroDdelta}
(2\delta)^{1/N}(1-\eps)B(P_\delta)\cap\Omega \subset D_\delta^{\text{in}} \subset (2\delta)^{1/N}(1+\eps)B(P_\delta)\cap\Omega.
\end{equation}
\end{enumerate}
\end{proposition}
\begin{proof}
For the first conclusion, let us assume by contradiction that 
$\delta_k\to0$ and $\delta_k^{1/2} \inf_{D_{\delta_k}}u_{\delta_k} =:\sigma_k\to0$. Then, taking the corresponding blow-up sequence as in Remark \ref{re:limpro} we obtain that, possibly up to subsequences, $t_k = \sigma_k\to0$, in contradiction with Lemma \ref{le:improvedconv}.

Analogously, for the second conclusion, assume by contradiction that 
$\delta_k\to0$ and, say, 
\[
x_k 
\in \Omega
\setminus
D_{\delta_k},
\qquad
x_k \in 
(2\delta_k)^{1/N}(1-\eps)B(P_k).
\]
Recalling \eqref{eq:Dk} and the fact that  $D\Psi(0)$ is the identity we obtain that 
\[
\frac{\Psi(x_k)-\Psi(P_k)}{\delta_k^{1/N}}\not\in \widetilde D_k,
\qquad
\frac{\Psi(x_k)-\Psi(P_k)}{\delta_k^{1/N}} \in 2^{1/N}(1-\eps')B
\]
for some $\eps'>0$ and $k$ large. This is in contradiction with 
\eqref{eq:BdentroDk}. The other inclusion can be obtained in the same way, considering a
sequence $x_k \in D_{\delta_k} \cap (2\delta_k)^{1/N}(1+\eps)\partial B(P_k)$.  
\end{proof}

Up to now we considered only 
global maximum points 
$P_\delta$. Now we are ready to deal also with local ones.
\begin{lemma}\label{lem:massimivicini}
Let $P'_\delta$ be any local maximum point for $u_\delta$. Then 
\[
\limsup_{\delta\to0}
\frac{|P_\delta-P'_\delta|}{\delta^{1/N}} = 0. 
\]
Equivalently, if $\delta_k\to0$ is 
such that $Q_k=\Psi(P_k)\to0$, then also $Q'_k:=\Psi(P'_k)\to0$.
\end{lemma}
\begin{proof}
We recall that $u_\delta\in C^{1,\alpha}(\overline{\Omega})$ 
is smooth and superharmonic in $D_\delta$, while it is smooth and subharmonic in $D_\delta^c$. By 
maximum principle and Hopf's lemma, we obtain that any local maximum point $P'_\delta\in\overline{D}_\delta$. Proposition \ref{prop:nonvanish} readily implies that
\begin{equation}\label{eq:no_vanish}
u_\delta (P'_\delta) \ge  
\sigma \delta^{-1/2}.
\end{equation}
for $\delta$ sufficiently small. This condition allows to repeat al the previous blow-up arguments, centered at $P'_\delta$, avoiding the vanishing of the blow-up limit.

\emph{Step 1: there exists $C>0$ such that, for $\delta$ small enough,
$
\dist(P'_\delta,\partial \Omega)\leq C\delta^{1/N}$.} The proof follows the lines of Lemma \ref{le:step1}, which contains the same result for global maximizers. According to Remark \ref{rem:maxglobvsloc} we only need to show that the limit of the blow-up sequence  $v'_j$, defined as in \eqref{eq:defscalings} but centered at $P'_j$, is non-trivial. This follows by pointwise convergence and as $v'_j(0)\ge \sigma$ for $j$ large, due to  \eqref{eq:no_vanish}.

\emph{Step 2: $\displaystyle\limsup_{\delta\to 0}\frac{\dist(P'_\delta,\partial \Omega)}{\delta^{1/N}}=0$.
} Again, we repeat the same argument of Lemma \ref{le:step2}, with the blow-up sequence $w'_k$, defined as in \eqref{eq:wk} but centered at $P'_k$, exploiting 
\eqref{eq:no_vanish} to deduce that its limit is non-trivial.

\emph{Step 3: there exists $C>0$ such that, for $\delta$ small enough, $|P_\delta-P'_\delta| \le
C\delta^{1/N}$.} Let us assume by contradiction that, for some sequence $\delta_k\to0$,
\begin{equation}\label{eq:distinfinita}
\lim_{k\to+\infty}\frac{|P_k-P_k'|}{\delta_k^{1/N}}=+\infty.
\end{equation}
Up to subsequences, and by step 2, this yields $P_k\to P\in\partial\Omega$, $P'_k\to P'\in \partial\Omega$, with $P'\neq P$. Let $w_k$ as in \eqref{eq:wk}, and let $w'_k$ be defined accordingly, centering the blow-up procedure 
at $P'_k$, with $\Psi'$ a diffeomorphism straightening the boundary near $P'$, and $Q'_k=\Psi'(P'_k)$. Proposition \ref{prop:nonvanish}, Remark \ref{re:limpro}, Lemma \ref{le:improvedconv} and equation \eqref{eq:BdentroDdelta} 
hold true also for $w'_\infty$, the blow-limit of $w'_k$. 

Now, taking $\eps$ sufficiently small in  \eqref{eq:BdentroDdelta}, we obtain, for $k$ large enough,
\[
(2\delta_k)^{1/N}(1-\eps)B(P_k)\cap\Omega\subset D_{\delta_k},
\qquad
|(2\delta_k)^{1/N}(1-\eps)B(P_k)\cap\Omega|\geq \frac34\delta_k.
\]
Analogously, 
\[
(2\delta_k)^{1/N}(1-\eps')B(P'_k)\cap\Omega\subset D_{\delta_k},
\qquad
|(2\delta_k)^{1/N}(1-\eps')B(P'_k)\cap\Omega|\geq \frac34\delta_k.
\]
Finally, by \eqref{eq:distinfinita} 
\[
(2\delta_k)^{1/N}(1-\eps)B(P_k)
\cap
(2\delta_k)^{1/N}(1-\eps')B(P'_k)
=\emptyset.
\]
Summing up we obtain the contradiction
\[
|D_{\delta_k}|\geq 2\cdot\frac34\delta_k = \frac32|D_{\delta_k}|,\qquad\text{for all $k$  sufficiently
large}.
\]

\emph{Step 4: conclusion.} Again
by contradiction, let $\delta_k\to0$ and 
\[
0<\lim_{k\to+\infty}\frac{|P_k-P_k'|}{\delta_k^{1/N}} \le C
\]
by the previous step. Considering 
the usual blow-up sequence (centered at $P_k$) we have that 
the function $w_k$ has two local maxima, one for $z=0$ and  one for $z=z_{k}%=\frac{Q_k'-Q_k}{\delta_k^{1/N}}
=(Q_k'-Q_k)\delta_k^{-1/N}
\le C$. Then, up to subsequences, 
$z_k\to z_\infty\neq 0$, by the contradiction assumption. This is not possible, since by uniform convergence $z_\infty$ is a local maximum for $w_\infty$, i.e.\ $z_\infty=0$. 
\end{proof}
To conclude the proof of Theorem \ref{th:intro_D} we prove the following. 
\begin{proposition}\label{prop:unsolomax}
There exists $\delta_{0}>0$ such that for every $\delta\in (0,\delta_{0})$:
\begin{enumerate}
\item $u_{\delta}$ has a unique local maximum point $P_{\delta}$,
\item $P_{\delta}\in \partial \Omega$, 
\item $D_\delta$ is connected.
\end{enumerate}
\end{proposition}
\begin{proof}
It suffices to prove that $w_k$ has a unique local 
maximum point in $B_R$, for some $R>0$ and $k$ large. Indeed, in such a case, 
Lemma \ref{lem:massimivicini} already excludes the presence of other local maxima for $u_\delta$, outside $B_R$ and 1 follows; on the other hand, in case $P_k\in\Omega$ then $w_k$ would have two distinct maxima in $B_R$, by reflection (recall Lemma \ref{le:step2} and \eqref{eq:wk}), and therefore 2 follows; finally, $u_\delta$ has a local maximum in each connected component 
of $D_\delta$, and also 3 follows.

To prove the claim, let $R$ be such that $\overline{B}_R\subset 2^{1/N}B$. By Remark \ref{rem:C2} 
$w_k\to w_\infty$ 
in $C^2(\overline{B}_R)$; moreover, since $w_\infty''(0)<0$ we can take $R$ small so that $w_\infty''<0$ 
in $B_R$. Recalling that $w_k$ achieves its maximum at the origin (hence $\nabla w_k(0)=0$), and 
applying~\cite[Lemma~4.2]{nitakagi_cpam}, we deduce that $\nabla w_k(z)\not=0$ for all $z\in 
\overline B_R\setminus \{0\}$, thus the origin is the unique maximum point of $w_k$ in $B_R$ for 
$k$ sufficiently large. 
\end{proof}
\begin{remark}\label{rem:Hn-1}
By Proposition \ref{prop:unsolomax}, we have that the second conclusion of Proposition 
\ref{prop:nonvanish} holds for $D_\delta$, instead of $D_\delta^{\text{in}}$. In turn, recalling that $\partial \Omega$ is regular and $P_\delta\in\partial\Omega$, 
we deduce that
\[
{\mathcal H}^{N-1}(\partial D_{\delta}\cap \partial \Omega)\ge 
{\mathcal H}^{N-1}((2\delta)^{1/N}(1-\eps)B(P_\delta) \cap \partial \Omega)\ge C\delta^{(N-1)/N},
\]
where 
\[
0<C<\omega_{N-1}2^{(N-1)/N}\omega_N^{-(N-1)/N}
\] 
and  $\delta$ is sufficiently small.
\end{remark}

We conclude this section with some estimates about the asymptotical decay of the eigenfunction $u_{\delta}$. Such estimates can be obtained as a byproduct of the previous blow-up analysis. We 
refer the reader  to \cite[Theorem ~2.3]{nitakagi_cpam} for more details.

\begin{proposition}\label{prop:2.3}
With the notation above, for all $\eta>0$ there exists $\delta_0>0$ and a subdomain $\Omega_\delta^{(i)}\subset \Omega$ such that, for all $\delta\in (0,\delta_0)$ we have:
\begin{enumerate}
\item[(i)] $P_\delta\in \partial \Omega_\delta^{(i)}$ and $\diam(\Omega_\delta^{(i)})\leq \widehat C\delta^{1/N}$,
\item[(ii)] $\|\delta^{1/2}u_\delta(\cdot)-w(\Psi(\cdot)/\delta^{1/N})\|_{C^{1,\vartheta}(\overline \Omega_\delta^{(i)})}\leq \eta$,
\item[(iii)] $|u_\delta(x)|\leq C_1 \delta^{-\frac12}\eta e^{-\mu_1d(x)/\delta^{1/N}}$, for all $x\in \Omega\setminus \Omega_\delta^{(i)}$,
\end{enumerate}
where $d(x):=\min\{\dist(x,\partial \Omega_\delta^{(i)}),\eta_0\}$ and $\widehat C, C_1, \mu_1,\eta_0$ are positive constant depending only on $\Omega$.
\end{proposition}

\section*{Acknowledgments} 
Work partially supported by the PRIN-2017-JPCAPN Grant: ``Equazioni 
differenziali alle de\-ri\-va\-te parziali non lineari'',
by the project Vain-Hopes within the program VALERE: VAnviteLli pEr la RicErca,
by the Portuguese
government through FCT/Portugal under the project PTDC/MAT-PUR/1788/2020, 
and by the INdAM-GNAMPA group.

%\bibliography{frac_eig}

\begin{thebibliography}{10}

\bibitem{MR3498523}
H.~Berestycki, J.~Coville, and H.-H. Vo.
\newblock Persistence criteria for populations with non-local dispersion.
\newblock {\em J. Math. Biol.}, 72(7):1693--1745, 2016.

\bibitem{MR2214420}
H.~Berestycki, F.~Hamel, and L.~Roques.
\newblock Analysis of the periodically fragmented environment model. {I}.
  {S}pecies persistence.
\newblock {\em J. Math. Biol.}, 51(1):75--113, 2005.

\bibitem{MR929981}
J.~E. Brothers and W.~P. Ziemer.
\newblock Minimal rearrangements of {S}obolev functions.
\newblock {\em J. Reine Angew. Math.}, 384:153--179, 1988.

\bibitem{MR1014659}
R.~S. Cantrell and C.~Cosner.
\newblock Diffusive logistic equations with indefinite weights: population
  models in disrupted environments.
\newblock {\em Proc. Roy. Soc. Edinburgh Sect. A}, 112(3-4):293--318, 1989.

\bibitem{MR1105497}
R.~S. Cantrell and C.~Cosner.
\newblock The effects of spatial heterogeneity in population dynamics.
\newblock {\em J. Math. Biol.}, 29(4):315--338, 1991.

\bibitem{MR2191264}
R.~S. Cantrell and C.~Cosner.
\newblock {\em Spatial ecology via reaction-diffusion equations}.
\newblock Wiley Series in Mathematical and Computational Biology. John Wiley \&
  Sons, Ltd., Chichester, 2003.

\bibitem{dipierro2021nonlocal}
S.~Dipierro, E.~P. Lippi, and E.~Valdinoci.
\newblock ({N}on)local logistic equations with {N}eumann conditions, 2021.
\newblock Preprint, arXiv:2101.02315.

\bibitem{MR0058756vol2}
A.~Erd\'{e}lyi, W.~Magnus, F.~Oberhettinger, and F.~G. Tricomi.
\newblock {\em Higher transcendental functions. {V}ol. {II}}.
\newblock McGraw-Hill Book Company, Inc., New York-Toronto-London, 1953.
\newblock Based, in part, on notes left by Harry Bateman.

\bibitem{gilbargtrudinger}
D.~Gilbarg and N.~S. Trudinger.
\newblock {\em Elliptic partial differential equations of second order}.
\newblock Classics in Mathematics. Springer-Verlag, Berlin, 2001.
\newblock Reprint of the 1998 edition.

\bibitem{llnp}
J.~Lamboley, A.~Laurain, G.~Nadin, and Y.~Privat.
\newblock Properties of optimizers of the principal eigenvalue with indefinite
  weight and {R}obin conditions.
\newblock {\em Calc. Var. Partial Differential Equations}, 55(6):Paper No. 144,
  37, 2016.

\bibitem{lilo}
E.~H. Lieb and M.~Loss.
\newblock {\em Analysis}, volume~14 of {\em Graduate Studies in Mathematics}.
\newblock American Mathematical Society, Providence, RI, 1997.

\bibitem{ly}
Y.~Lou and E.~Yanagida.
\newblock Minimization of the principal eigenvalue for an elliptic boundary
  value problem with indefinite weight, and applications to population
  dynamics.
\newblock {\em Japan J. Indust. Appl. Math.}, 23(3):275--292, 2006.

\bibitem{mazari:hal-02987223}
I.~Mazari, G.~Nadin, and Y.~Privat.
\newblock {Some challenging optimisation problems for logistic diffusive
  equations and numerical issues}, 2020.
\newblock Preprint, hal-02987223.

\bibitem{mapeve}
D.~Mazzoleni, B.~Pellacci, and G.~Verzini.
\newblock Asymptotic spherical shapes in some spectral optimization problems.
\newblock {\em J. Math. Pures Appl. (9)}, 135:256--283, 2020.

\bibitem{mapeve_matrix}
D.~Mazzoleni, B.~Pellacci, and G.~Verzini.
\newblock Quantitative analysis of a singularly perturbed shape optimization
  problem in a polygon.
\newblock In {\em 2018 {MATRIX} annals}, volume~3 of {\em MATRIX Book Ser.},
  pages 275--283. Springer, Cham, [2020] \copyright 2020.

\bibitem{nitakagi_cpam}
W.-M. Ni and I.~Takagi.
\newblock On the shape of least-energy solutions to a semilinear {N}eumann
  problem.
\newblock {\em Comm. Pure Appl. Math.}, 44(7):819--851, 1991.

\bibitem{nitakagi_duke}
W.-M. Ni and I.~Takagi.
\newblock Locating the peaks of least-energy solutions to a semilinear
  {N}eumann problem.
\newblock {\em Duke Math. J.}, 70(2):247--281, 1993.

\bibitem{MR3771424}
B.~Pellacci and G.~Verzini.
\newblock Best dispersal strategies in spatially heterogeneous environments:
  optimization of the principal eigenvalue for indefinite fractional {N}eumann
  problems.
\newblock {\em J. Math. Biol.}, 76(6):1357--1386, 2018.

\bibitem{haro}
L.~Roques and F.~Hamel.
\newblock Mathematical analysis of the optimal habitat configurations for
  species persistence.
\newblock {\em Math. Biosci.}, 210(1):34--59, 2007.

\end{thebibliography}
%\bibliographystyle{abbrv}
%

\medskip
\small
\begin{flushright}
\noindent \verb"dario.mazzoleni@unipv.it"\\
Dipartimento di Matematica  ``F. Casorati'', \\
Universit\`a di Pavia\\
via Ferrata 5, 27100 Pavia, Italy\\
\medskip
\noindent \verb"benedetta.pellacci@unicampania.it"\\
Dipartimento di Matematica e Fisica,\\
Universit\`a della Campania  ``Luigi Vanvitelli''\\
via A. Lincoln 5, Caserta, Italy\\
\medskip
\noindent \verb"gianmaria.verzini@polimi.it"\\
Dipartimento di Matematica, Politecnico di Milano\\
piazza Leonardo da Vinci 32, 20133 Milano, Italy\\
\end{flushright}

\end{document}